\documentclass[10pt,reqno]{amsart}
\usepackage{amssymb}
\usepackage{amsmath}
\usepackage{amsthm}
\usepackage{graphicx}

\usepackage{subfigure}

\usepackage{caption}
\usepackage{bm}
\usepackage[colorlinks=true,urlcolor=blue,
citecolor=red,linkcolor=blue,linktocpage,pdfpagelabels,
bookmarksnumbered,bookmarksopen]{hyperref}%per colore-link
\usepackage[titletoc,title]{appendix}
\numberwithin{equation}{section} %riparte da zero ogni sezione
\newcounter{cont}[section] 

\newtheorem{thm}[cont]{Theorem}
\newtheorem{prop}[cont]{Proposition}
\newtheorem{lem}[cont]{Lemma}

\theoremstyle{definition}
\newtheorem{defn}[cont]{Definition}
 \theoremstyle{remark}
 \newtheorem{rem}[cont]{Remark}

\newcommand{\R}{\mathbb{R}}
\newcommand{\e}{\varepsilon}

\newcommand{\N}{\mathbb{N}}

\begin{document}

\title[Dynamics of solutions to $p$-Laplacian reaction diffusion equations]{Long time dynamics of solutions to $p$-Laplacian diffusion problems with bistable reaction terms}

\author[R. Folino]{Raffaele Folino}

\address[R. Folino]{Departamento de Matem\'aticas y Mec\'anica\\Instituto de 
Investigaciones en Matem\'aticas Aplicadas y en Sistemas\\Universidad Nacional Aut\'onoma de 
M\'exico\\Circuito Escolar s/n, Ciudad Universitaria C.P. 04510 Cd. Mx. (Mexico)}

\email{folino@mym.iimas.unam.mx}

\author[R.G. Plaza]{Ram\'on G. Plaza}

\address[R. G. Plaza]{Departamento de Matem\'aticas y Mec\'anica\\Instituto de 
Investigaciones en Matem\'aticas Aplicadas y en Sistemas\\Universidad Nacional Aut\'onoma de 
M\'exico\\Circuito Escolar s/n, Ciudad Universitaria C.P. 04510 Cd. Mx. (Mexico)}

\email{plaza@mym.iimas.unam.mx}

\author[M. Strani]{Marta Strani}

\address[M. Strani]{Dipartimento di Scienze Molecolari e Nanosistemi\\Universit\`a Ca' Foscari Venezia Mestre\\Campus Scientifico\\Via Torino 155, 30170 Venezia Mestre (Italy)}

\email{marta.strani@unive.it}

\keywords{$p$-Laplacian; reaction diffusion equations; transition layer structure; metastability; energy estimates}

\subjclass[2010]{35K91, 35K57, 35B36, 35B40, 35K59}

\maketitle

%%%%%%\nocite{*}

\begin{abstract} 
This paper establishes the emergence of slowly moving transition layer solutions for the $p$-Laplacian (nonlinear) evolution equation,
\[
u_t = \e^p(|u_x|^{p-2}u_x)_x  - F'(u), \qquad x \in (a,b), \; t > 0,
\] 
%$x \in (a,b)$, $t > 0$, 
where $\e > 0$ and $p{>1}$ are constants,  driven by the action of a family of double-well potentials of the form
\[
F(u)={\frac{1}{2n}|1-u^2|^{n}},
\]
{indexed by $n>1$, $n \in \R$} with minima at two pure phases $u = \pm 1$. 
The equation is endowed with initial conditions and boundary conditions of Neumann type. It is shown that interface layers, or solutions which initially are equal to $\pm 1$ except at a finite number of thin transitions of width $\e$, persist for an exponentially long time in the critical case with $n=p$, and for an algebraically long time in the supercritical (or degenerate) case with $n>p$. For that purpose, energy bounds for a renormalized effective energy potential of Ginzburg-Landau type are established. In contrast, in the subcritical case with $n< p$, the transition layer solutions are stationary. 
\end{abstract}

%\tableofcontents

\section{Introduction}\label{sec:intro}
Reaction-diffusion equations are broadly used to describe common phenomena such as pattern formation
and front propagation in biological, chemical and physical systems.
In their one-space dimensional form,  \emph{reaction-diffusion equations} read as
\begin{equation}\label{readiff-gen}
	 u_t=D  u_{xx}+ f( u), \qquad  x\in I\subset \R, \quad t>0,
\end{equation}
where $ u= u( x,t)\in \R$, the constant $D>0$ is the diffusion coefficient, and the reaction term $ f:\R\rightarrow\R$ is a smooth function.
Equations of the type \eqref{readiff-gen} are basic models describing competition between the diffusion term, namely $Du_{xx}$, and the reaction term, $ f(u)$. 
The combination of a linear diffusion together with a nonlinear interaction term
produces mathematical features that are not predictable by looking at one of the two mechanism alone;       
indeed, the term $Du_{xx}$ acts in such a way as to ``spread uniformly'' the solution $u$,
while the dynamics of $ u_t= f( u)$ can produce different phenomena as, for example, large solutions and step gradients, and this leads to the possibility of critical behaviors.      
This can be clearly observed if one considers the balanced bistable reaction term $f(u)=u(1-u^2)$, 
as in the classical  \emph{Allen--Cahn equation},
\begin{equation}\label{eq:Al-Ca}
	u_t=\e^2 u_{xx}+u(1-u^2), \qquad \quad 0 < \e \ll 1,
\end{equation} 
introduced in 1979 by  S. M. Allen and J. W. Cahn \cite{AlCa79} to describe the interface motion between different crystalline structures in alloys. 
To be more precise, the dynamics of solutions to \eqref{eq:Al-Ca} involves
two different effects: the reaction term $f$ pushes
the solution towards $u=\pm1$ (stable zeros of $f$), while the diffusion term tends to regularize and smoothen the solution. When the diffusion coefficient $\e$ tends to zero, two different phases appear, corresponding
to intervals where the solution is close to either $u=+1$ or $u=-1$, and the width of the transition layers between these two phases is of order $\e$. In this scenario, a peculiar phenomenon occurs: interface layers are transient solutions that appear to be stable, but which, after an \emph{exponentially} long time of order $\mathcal{O}(e^{1/\e})$, drastically change their shape and converge to one of the equilibrium states. This phenomenon is known in the literature as \emph{metastability}, and it was first studied in the context of the classical Allen--Cahn equation in the pioneering works of Carr and Pego \cite{CaPe89,CaPe90}, Fusco and Hale \cite{FuHa89} and Bronsard and Kohn \cite{BrKo90}, which appeared approximately at the same time and which applied different methodologies (see also \cite{ChenX04}). Since then, the metastability of transition layer structures has been studied in (and extended to) many other models such as hyperbolic equations \cite{Fol17,Fol19,FLM19}, parabolic systems \cite{Str14}, gradient flows \cite{OtRe07}, viscous conservation laws \cite{FLMS17,LaOM95, MaSt13, ReWa95} or reaction-diffusion equations with phase-dependent diffusivities \cite{FHLP}, just to mention a few.

\vskip0.2cm

In this paper we analyze the dynamics of the solutions to \eqref{readiff-gen} when the classical linear diffusion is replaced by the nonlinear \emph{$p-$Laplacian operator} $(|u_x|^{p-2}u_x)_x$, while the reaction term $f$ is chosen in such a way to  generalize  the one in \eqref{eq:Al-Ca}; more precisely, our aim is to study the long time dynamics of the solutions to the following reaction diffusion model 
%with a $p-$Laplacian operator and a bistable reaction term 
\begin{equation}\label{eq:P-model}
	u_t=\e^p(|u_x|^{p-2}u_x)_x+{u(1-u^2)|1-u^2|^{n-2}},
\end{equation}
where $u=u(x,t) : [a,b]\times(0,+\infty)\rightarrow \R$, $p {>1}$, {$n>1$}, $\e>0$ is a small parameter, and the reaction term $f(u)={u(1-u^2)|1-u^2|^{n-2}}$ is a prototype of bistable function with two stable zeros at $u=\pm1$ and an unstable one at $u=0$. In particular, the reaction term can be interpreted as the force, $f=-F'$, deriving from a potential $F:\R\to\R$ given by
\begin{equation}\label{defF}
	F(u)={\frac{1}{2n}|1-u^2|^{n}, \quad n>1}.
\end{equation}
Being the exponent $n$ in \eqref{defF} {greater than or equal to $1$}, $F$ is a double well potential with wells of equal depth in $u=\pm1$. Thus, equation \eqref{defF} describes a family (indexed by {$n>1$, $n \in \R$}) of double-well potentials which underly force (or reaction) terms of bistable type. Figure \ref{figPotential} below shows the shapes of these potentials for different values of $n ={2,4,6}$, exhibiting the double-well structure. The classical Allen--Cahn equation \eqref{eq:Al-Ca} is a particular case (obtained by choosing $p=2$ and $n={2}$) of model \eqref{eq:P-model}.

\begin{rem}
Instead of the explicit form \eqref{defF}, one may consider a generic potential $F\in C^1(\R)$ satisfying 
\begin{equation}\label{eq:ass-F}
\begin{aligned}
	&F(\pm1)=F'(\pm1)=0, \qquad \qquad F(u)>0 \quad \forall\, u\neq\pm1,  \\
	&{\lambda_1 |1\pm u|^{n} \leq F(u) \leq \lambda_2 |1\pm u|^{n}, \qquad\qquad \mbox{ for } \qquad |u\pm1|<\eta,}
	\end{aligned}
\end{equation}
{for some constants $0<\lambda_1\leq\lambda_2$, $\eta>0$ and $n>1$,} 
and analyze equation \eqref{eq:P-model} with reaction term  $f(u)=-F'(u)$. 
We claim that the theory developed in this paper can be extended to such a case with not much extra bookkeeping needed.
\end{rem}

\begin{figure}[hbtp]
\centering
\includegraphics[scale=.6, clip=true]{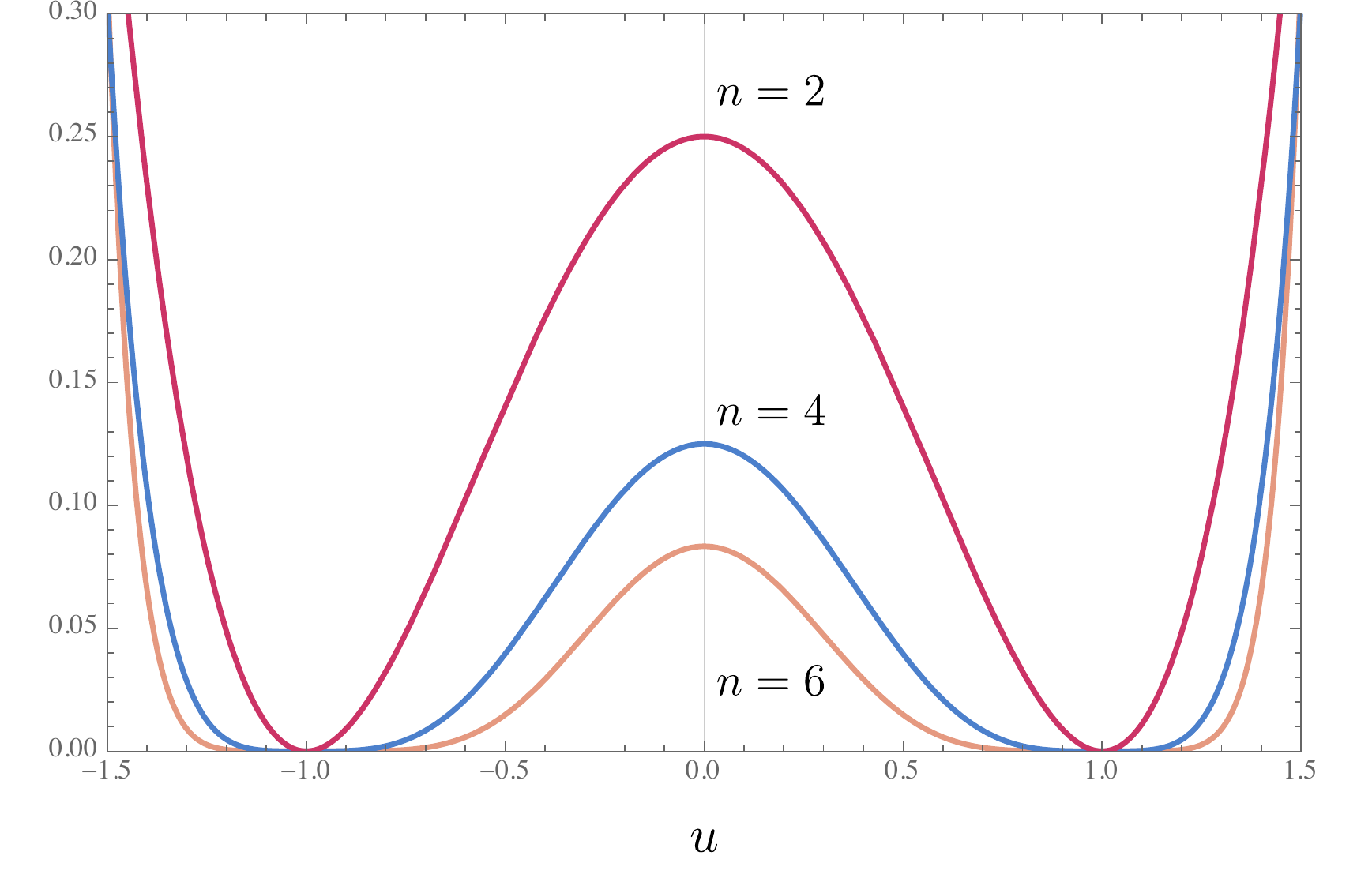}
\caption{\small{Plots of the potential function \eqref{defF} for $n={2,4,6}$, which underly different behaviors when compared to the diffusion parameter $p \geq 2$. 
For example, when $p = 4$ the former cases correspond to subcritical ($n < p$), critical ($n=p$) and supercritical or degenerate ($n>p$) cases, respectively.}}
\label{figPotential}
\end{figure}

According to custom in the study of phase coexistence models, we consider equation \eqref{eq:P-model} complemented with homogeneous Neumann boundary conditions
\begin{equation}\label{eq:Neu}
	u_x(a,t)=u_x(b,t)=0, \qquad \qquad t>0,
\end{equation}
and initial datum
\begin{equation}\label{eq:initial}
	u(x,0)=u_0(x), \qquad \qquad x\in[a,b].
\end{equation}

Historically, the $p-$Laplacian operator first appeared from a power law alternative to Darcy's phenomenological law to describe fluid flow through porous media (see, for instance, the recent review paper by Benedikt \emph{et al.} \cite{BGKT18} and the references therein). Since then, the $p-$Laplacian has established itself as a fundamental quasilinear elliptic operator and has been intensely studied in the last fifty years. A comprehensive list of only basic analytical results for the $p-$Laplacian would have to contain hundreds of references and that is not our purpose here. The reader is referred to the classical book by Lions \cite{LionsJL69} for a modern functional analytic treatise for the $p-$Laplacian and related quasilinear (elliptic or evolution) equations, as well as to the recent monograph by Lindqvist \cite{Lind19} for the stationary equation. Relatively less attention has been paid to the evolution (parabolic) $p-$Laplacian equation, even though the literature is also very extensive. An abridged list of references include \cite{DiB93,LPV06,LionsJL69,Lou03,TY00}. Up to our knowledge, the $p-$Laplacian, interpreted as a diffusive mechanism in reaction-diffusion models, has not been studied in the context of long time behavior (metastability) of phase transition layer solutions. What is the effect of the $p-$Laplace operator on the behavior of interface layer solutions to a basic model like \eqref{eq:P-model}? What is the interplay between diffusion and the double-well potential under consideration? In this paper, we provide comprehensive and detailed answers to these questions.

Henceforth, our main objective is to investigate the behavior of the solutions to \eqref{eq:P-model}-\eqref{eq:Neu} for large times, 
enlightening, in particular, the interaction between the constant $p$ (appearing in the diffusion term) and the constant $n$ 
(characterizing the behavior of the potential \eqref{defF}). 
More precisely, we study three different situations:
\vskip0.2cm
\begin{itemize}
\item[(a)] If ${1<n}<p$, then there exist stationary solutions to \eqref{eq:P-model}-\eqref{eq:Neu} with a $N$ transition layer structure inside the interval $[a,b]$ 
(see section \ref{sub:SS} below); 
{the main examples of such a case are given by the classical Laplacian operator $p=2$ and $1<n<2$ in \eqref{eq:P-model}},
and the $p$-Laplacian evolution equation \eqref{eq:P-model} with $p>2$ and $n={2}$, 
corresponding to the usual double well potential with two equal minima in $u=\pm 1$. We call this the \emph{subcritical case}.
\vskip0.2cm
\item[(b)] If ${n=p>1}$, then (as in the linear case $p=2$) equation \eqref{eq:P-model} exhibits the phenomenon of metastability (see section \ref{sec:2n=p}). 
More precisely, the solutions arising from an initial datum $u_0$ with a \emph{transition layer structure} (see Definition \ref{def:TLS}) 
will maintain such a structure for an exponentially long time, that is a time of $\mathcal{O}(e^{Ap/2\e})$, for some $A>0$, 
and after that they will converge to one of the minimal points of the potential $F$. 
We call this the \emph{critical case}.
\vskip0.2cm
\item[(c)] If ${n>p>1}$, then the potential $F$ satisfies $F^{([p])}(\pm 1)=0$ (where $[p]$ denotes the integer part of $p$)
and solutions with an $N$-transition layer structure still maintain their shape for a long time, 
but the order of such persistence is only algebraic in $\e$, precisely of the order $\mathcal{O}(\e^{-\gamma_{n,p}})$ for some positive $\gamma_{n,p}$. 
Hence no metastability is observed. 
We call this the \emph{supercritical or degenerate case}.
\end{itemize}
\vskip0.2cm

By convention, we have chosen the term supercritical or degenerate to characterize potentials which are degenerate \emph{with respect to the $p-$Laplacian operator}, satisfying $F^{([p])}(\pm 1)=0$ (notice that this always happens if $1<p<2$, but it means that more derivatives of $F$ vanish at $\pm1$ if $p\geq2$). This {nomenclature} is consistent with the concept of a degenerate double-well potential with respect to the classical Laplace operator (see, e.g., \cite{DFV18}). In general (also in the case $1<p<2$),  it describes an energy regime in which the algebraic power of the potential, ${n}$, exceeds the diffusion parameter $p$. For instance, Figure \ref{figPotential} shows different potential functions \eqref{defF} for values $n = {2,4,6}$. These potentials may or may not be degenerate depending on the diffusion parameter $p$ under consideration. The shape of the potential with $n={2,4,6}$ in Figure \ref{figPotential}, for example, may seem degenerate with respect to the classical diffusion, but it is not so when $p$ is large enough ($p > 6$). In several space dimensions, this definition of degeneracy (or criticality) should depend as well on the dimension $d$ of the physical space (see \cite{DFV18,PeV05} for further information). For a study of variational properties of subcritical (according to our definition) potentials with respect to the $p-$Laplacian, the reader is referred to \cite{CLY09} (see also the recent work \cite{HuSo20}).

The goal of this paper is to prove the behaviors (a), (b) and (c) described above, paying particular attention to the critical and degenerate cases where ${n} \geq p$. 
Indeed, as we have already mentioned, in the subcritical case ${1<n} < p$ we will see that there exist solutions with a $N$-transition layer structure for the problem (hence invariant under the dynamics of \eqref{eq:P-model}), independently of the choice of $N \in \mathbb{N}$ (for the precise statement of such a claim, see Proposition \ref{prop:2n<p} below). As a consequence, if one starts with an initial configuration with such a structure,  the corresponding time-dependent solution will clearly satisfy $u(x,t) \equiv u_0(x)$ for all $t>0$. In other words, transition layers do not evolve in time and persist forever. An interesting related question is whether these structures are dynamically stable under small perturbations, inasmuch as it has been recently proved that they are unstable as variational solutions to the associated elliptic problem (see Theorem 1.5 in \cite{DPV20}). The dynamical (in)stability of these stationary solutions will be addressed in a companion paper.

In contrast, if ${n} \geq p$ then solutions to \eqref{eq:P-model} starting with an $N$-transition layer structure {\it slowly converge} towards one of the minima of the potential $F$ and the time of such convergence drastically changes among the two cases. For instance, in the critical case where ${n}=p$ we prove the {\it exponentially slow motion} of the solutions to \eqref{eq:P-model} exhibiting, in this fashion, the phenomenon of metastability. More precisely, we show that the solutions evolve so slowly that they \emph{appear} to be stable,
and it is only after a very long time that they converge to their asymptotic limit (the constant profile with values $+1$ or $-1$).
It is to be noticed that the same behavior is observed in the classical {Allen--Cahn equation} \eqref{eq:Al-Ca},
that is, when $p=2$ and $n={2}$  in \eqref{eq:P-model},
equipped with homogeneous Neumann boundary conditions \eqref{eq:Neu}. Hence, our analysis recovers the classical results in this case \cite{CaPe89}.
In particular, inspired by the paper of Bronsard and Kohn \cite{BrKo90}, where the authors study metastable properties to \eqref{eq:Al-Ca} 
by means of an \emph{energy approach} (based on the fact that the Allen--Cahn equation can be seen as a gradient flow in $L^2(a,b)$ of the energy associated to the system), 
we apply a similar strategy and define the functional,
\begin{equation}\label{energyintro}
	E[u]=\int_a^b\left[\frac{\e^p|u_x|^p}{p}+F(u)\right]\,dx,
\end{equation}
which is an energy functional of Ginzburg-Landau type associated to the $p$-Laplacian diffusion equation \eqref{eq:P-model}. The strategy of \cite{BrKo90}, albeit very powerful, leads the authors to establish {\it algebraic slow motion} of the solutions to \eqref{eq:Al-Ca}. Grant \cite{Grnt95} showed that the energy  approach is capable of obtaining an {\it exponentially slow motion}, as he demonstrated it in the case of solutions to Cahn--Morral systems. 
For further developments and applications of the energy method we quote, among others, \cite{Fol17,Fol19,FLM19} for hyperbolic systems, and, more recently, the paper
 \cite{FPS1}, for reaction diffusion equations with mean curvature type diffusions, 
 to which we refer the reader in search of a brief summary of the original strategy of Bronsard and Kohn \cite{BrKo90}.

Motivated by all these previous results, in this paper we adapt the energy approach to the IBVP \eqref{eq:P-model}-\eqref{eq:Neu}-\eqref{eq:initial}. 
The main idea of  \cite{BrKo90} is to use a renormalized energy which, in our case, is defined from \eqref{energyintro} as
\begin{equation}\label{eq:energy}
	E_\e[u]= \frac{E[u]}{\e}=\int_a^b\left[\frac{\e^{p-1}|u_x|^p}{p}+\frac{F(u)}{\e}\right]\,dx.
\end{equation}
The key point is to prove that for any function $u$ sufficiently close in $L^1(a,b)$ to a step function $v$, the following inequality holds
\begin{equation*}
	E_\varepsilon[u]\geq Nc_p-C\exp(-Ap/2\varepsilon),
\end{equation*}
for some $A>0$ and $c_p>0$.  
Such a lower bound is crucial to prove the persistence, for an exponentially long time, of metastable patterns for \eqref{eq:P-model}; 
indeed, this variational result, together with the fact that, if $u$ is a solution to \eqref{eq:P-model} with homogeneous Neumann boundary conditions \eqref{eq:Neu}, then 
\begin{equation}\label{eq:energy-var-intro}
	\frac{d}{dt} E[u](t)=-\int_a^b u_t^2(x,t)\,dx,
\end{equation}
allows us to proceed as in \cite{BrKo90,Grnt95} and to prove the \emph{exponentially slow motion} of the solutions to \eqref{eq:P-model}-\eqref{eq:Neu}. 
This result extends the classical ones on \eqref{eq:Al-Ca} to the case of the $p$-Laplacian model \eqref{eq:P-model}-\eqref{defF}  for any  ${n} = p{>1}, \ {n \in \R}$. 
However, while the exponentially slow motion for the classical  Allen--Cahn equation ($p=2$) was already proved, even if with different strategies 
(see, for instance \cite{CaPe89}),  the study in the case of a purely $p$-Laplacian operator is, to the best of our knowledge, new.

\vskip0.2cm
On the other hand, when ${n>p>1}$ we prove that the time taken by the solutions to \eqref{eq:P-model} to reach one of the two constant solutions $u=\pm 1$ (minima of $F$) is only {algebraical in $\e$}; 
again, the key point to achieve such result is a lower bound on the energy that  in this case reads
\begin{equation*}
	E_\e[u]\geq Nc_p-C\varepsilon^{\gamma_{n,p}},
\end{equation*}
where $\gamma_{n,p}>0$ depends both on $n$ and $p$. 
Hence, we still have a slow motion of the solutions towards the equilibria 
(we will see that, for small $\e$, unstable patterns persist for times of the order $\mathcal O(\e^{-\gamma_{n,p}})$), but no metastability occurs. 
Such behavior, which is typical of the degenerate case, has already been observed on the whole line when $p=2$ 
by Bethuel and Smets in \cite{BetSme2013} (in this case, the degeneracy translates into the fact that $F''(\pm 1)=0$).
It is worth noticing that the exponent $\gamma_{n,p}$  we obtain here (for more details see Section \ref{sec:degcase}), {coincides with} the one of \cite{BetSme2013} 
in the case $p=2$. 
In this spirit, our result is, to the best of our knowledge, the first result proving the (algebraical) slow motion of solutions to reaction diffusion problems in bounded intervals in the degenerate case, also in the case of the classical Laplacian operator.

\vskip0.2cm
We close this Introduction by sketching the plan of the paper. In Section \ref{sec2} we study the existence of stationary solutions to \eqref{eq:P-model}, analyzing at first the problem on the whole line. Indeed, steady states on the bounded interval $[a,b]$ and satisfying the boundary conditions \eqref{eq:Neu} can be explicitly constructed from stationary solutions to the equation on the whole real line (see Propositions \ref{prop:2n<p} and \ref{prop:2n>p} below). We will see that, as it happens for the time dependent problem, there are substantial differences among the subcritical, ${n} < p$, and the critical and degenerate cases with ${n} \geq p$, the more interesting being the fact that in the first one stationary solutions can oscillate between the values $u=\pm 1$ by touching them, while in the latter this is not possible.

Once the problem of existence of stationary solutions is understood, we focus our attention on the time {dependent} solutions to \eqref{eq:P-model}. In Section \ref{sec:2n=p} we analyze the dynamics in the case ${n} = p$, showing that some solutions to \eqref{eq:P-model} exhibit the phenomenon of metastability. More precisely, we show that solutions starting with an initial configuration with $N$-transition layers will maintain such a structure for  an exponentially long time of order $\mathcal O(e^{c p /\e})$, $c>0$ (see Theorem \ref{thm:main}), before converging to their asymptotic limit.

Finally, in Section \ref{sec:degcase} we study the degenerate case ${n} > p$. As we have already mentioned, the solutions still maintain an unstable structure with $N$-transition layers for long times, but in this case the convergence to the equilibria happens for times which are algebraical with respect to $\e$, more precisely, for $t = \mathcal O(\e^{-\gamma_{n,p}})$ (see Theorem \ref{thm:main2} below).

\vskip0.2cm
Let us stress once again that the exponentially slow motion obtained in Section \ref{sec:2n=p} for the critical case constitutes an extension to the $p$-Laplacian operator, ${p>1}$, of the results already known for the classical Laplacian;  {in particular, our analysis  proves metastable dynamics also when $1<p<2$ and $n=p$.}  Moreover, the estimate of the time of convergence of the solutions is sharp since it reads exactly as the classical one in \cite{CaPe89} when $p = 2$. On the other hand, the algebraic slow motion in the supercritical (or degenerate) case obtained in Section \ref{sec:degcase} is, to the best of our knowledge, new, even in the case of classical reaction diffusion equations: indeed, the only {known} result in this direction is given by \cite{BetSme2013}, where the same problem is addressed on the whole real line, and it is worth noticing that the algebraical order of convergence obtained (see Remark \ref{sharp}) is sharp {since it reads as the estimate in \cite{BetSme2013} if $p=2$}.

\section{Stationary solutions on the real line and on bounded intervals}\label{sec2}
The aim of  this section is to study the stationary problem associated to the equation \eqref{eq:P-model}, i.e. to describe the solutions to
\begin{equation}\label{SS}
	\e^{p}\left(|\varphi'|^{p-2}\varphi'\right)'-F'(\varphi)=0, \qquad {p>1,} \qquad F(\varphi)={\frac{1}{2n}|1-\varphi^2|^{n}, \quad n>1,}
\end{equation}
both on the whole real line and on a bounded interval.
We will see that there is a substantial difference between the subcritical (${1<n}<p$) and the critical and degenerate (${1<p\leq n}$) cases. 

\subsection{Stationary solutions on the real line}\label{sub:SW}
We start by analyzing problem \eqref{SS} on the whole real line. 
In particular, in this section we focus the attention on two kinds of bounded non constant stationary solutions: 
standing waves (monotone) and periodic solutions (non monotone).

\subsubsection{Standing waves}\label{substanding}
We look for standing waves solutions to \eqref{SS}, i.e. monotone solutions $\Phi_\e=\Phi_\e(x)$ to the boundary value problem
\begin{equation}\label{eq:Fi}
	\e^p\left((\Phi'_\e)^{p-1}\right)'-F'(\Phi_\e)=0, \qquad \lim_{x\to\pm\infty}\Phi_\e(x)=\pm1, \qquad \Phi_\e(0)=0,
\end{equation}
where we erased the modulus since we consider monotone {\it increasing} profiles.
\begin{prop}\label{prop:SW}
Let us consider the solution $\Phi_\e$ to \eqref{eq:Fi} with $F$ as in \eqref{SS}.
\begin{itemize}
\vskip0.2cm
\item[(i)] If ${1<n}<p$ then the profile $\Phi_\e$ reaches the states $\pm 1$ for a finite value of the $x$ variable. 
Precisely, there exists $\bar x\in(0,+\infty)$ such that $\Phi_\e(\pm  \, \bar x)= \pm 1$.
Moreover, $\Phi_\e\in C^\infty(\mathbb{R}\backslash\{\pm\bar x\})$ is a classical solution to \eqref{eq:Fi}.
\vskip0.2cm
\item[(ii)] If ${n}=p$ then the solution to \eqref{eq:Fi} is given by
\begin{equation}\label{eq:tanh}
	%\Phi_\e(x)= \frac{e^{\frac{2C_{n,p}}{\e} x} + 1}{1-e^{\frac{2C_{n,p}}{\e} x}  }= \tanh \left( \frac{C_{n,p} x}{\e} \right).
	\Phi_\e(x)=\tanh\left(\frac{C_p x}{\e}\right), \qquad \qquad C_p:=\left(\frac{1}{2(p-1)}\right)^{1/p},
\end{equation}
and we thus have an {\it exponential decay} of  $\Phi_\e$ towards the states $\pm 1$.
\item[(iii)] If ${n>p>1}$ then there exist $c_1,c_2>0$ (depending on $p$) such that
\begin{equation}\label{eq:alg-decay}
	\begin{aligned}
		&1-\Phi_\e(x)\leq c_1 x^{-c_2} \qquad &\mbox{as } x\to+\infty,\\ 
		&\Phi_\e(x)+1\leq-c_1x^{-c_2} &\mbox{as } x\to-\infty.
	\end{aligned}
\end{equation}
We thus have an {\it algebraic decay} of $\Phi_\e$ towards the states $\pm 1$.
\end{itemize}
\end{prop}
\begin{proof}
First we prove that there exists a unique solution to \eqref{eq:Fi}.
Multiplying by $\Phi'_\e$ the first equation of \eqref{eq:Fi} and using the equality 
$$\left((\Phi'_\e)^{p-1}\right)'\Phi'_\e=\left[\frac{(p-1)}{p}(\Phi'_\e)^{p}\right]',$$ we get
\begin{equation*}
	\left[\frac{\e^{p}(p-1)}{p}(\Phi'_\e)^{p}-F(\Phi_\e)\right]'=0, \qquad\qquad \mbox{in }\, (-\infty,+\infty).
\end{equation*}
It follows that the profile $\Phi_\e$ satisfies 
\begin{equation}\label{eq:FI-first}
	\begin{cases}
		\e\Phi'_\e=\displaystyle{\left(\frac{p}{p-1}F(\Phi_\e)\right)^{1/p}},\\
		\Phi_\e(0)=0.
	\end{cases}
\end{equation}
Hence, by using the explicit expression of $F$ given in \eqref{defF}, there exists a unique solution to \eqref{eq:FI-first}  which is strictly increasing, and implicitly defined by
\begin{equation*}
	\int_{0}^{\Phi_\e(x)} \frac{ds}{{|1-s^2|^{n/p}}}=\left(\frac{p}{{2n}(p-1)}\right)^{1/p}\, \frac{x}{\e}.
\end{equation*} 
In order to prove properties (i), (ii) and (iii) of the statement, we thus have to study the convergence of the following improper integral
\begin{equation}\label{int_SW}
	\int_{0}^1 \frac{ds}{{|1-s^2|^{n/p}}},
\end{equation}
and such convergence depends on the value of the quantity ${n}/p$. 
First of all, if ${n}/p<1$ (that is, if ${n}<p$) then the integral \eqref{int_SW} converges, meaning that there exists $\bar x <+\infty$ such that $\Phi_\e(\bar x)=1$.
It is easy to check that $\Phi_\e\in C^\infty(\mathbb{R}\backslash\{\pm\bar x\})$ and, since \eqref{eq:FI-first} implies
$$(\Phi'_\e)^{p-1}=\e^{1-p}{\left(\frac{p}{p-1}F(\Phi_\e)\right)^{1-1/p}}\in C^1(\mathbb{R}),$$
we have that $\Phi_\e$ is a classical solution to \eqref{eq:Fi}.

On the other hand, the integral in \eqref{int_SW} diverges whenever ${n}\geq p$. 
In such a case, let us then consider the ODE in \eqref{eq:FI-first}, which reads
\begin{equation*}
	\e \Phi_\e' = C_{n,p} (1-\Phi_\e^2)^{n/p}, \quad \mbox{where} \quad  C_{n,p}:=  \left(\frac{p}{2n(p-1)}\right)^{1/p}.
\end{equation*}
In the case $n=p$, by separation of variables one obtains the explicit formula \eqref{eq:tanh};
finally, by defining $g(\Phi_\e):=  (1-\Phi_\e^2)^{n/p}$, we can compute
$$g'(\Phi_\e)=-\frac{{2}n}{p}\Phi_\e (1-\Phi_\e^2)^{\frac{n}{p}-1},$$
and observe that $g'(\pm 1) = 0$ if and only if $n>p$. 
The standard theory of ordinary differential equations leads to point (iii) of the statement and the proof is complete.
\end{proof}

We stress that, because of the symmetry of the potential $F$, it is easy to check that $-\Phi_\e(x)=\Phi_\e(-x)$ is a {\it decreasing} standing wave, namely a solution to
\begin{equation}\label{eq:Fi-}
	\e^p\left(|\Psi'_\e|^{p-2}\Psi'_\e\right)'-F'(\Psi_\e)=0, \qquad \lim_{x\to\pm\infty}\Psi_\e(x)=\mp1, \qquad \Psi_\e(0)=0,
\end{equation}
which satisfies the three cases of Theorem \ref{prop:SW} as well.
As a consequence,  since in the case $n<p$ we have $\Psi_\e(\mp\bar x)=\pm1$ for some $\bar x>0$, we can construct infinitely many (non monotone) solutions to
\begin{equation*}
	\e^p\left(|\Phi'_\e|^{p-2}\Phi'_\e\right)'-F'(\Phi_\e)=0, \qquad \lim_{x\to\pm\infty}\Phi_\e(x)=\pm1, \qquad \Phi_\e(0)=0.
\end{equation*}
On the other hand, since the integral \eqref{int_SW} diverges for $n\geq p$, in such a case
there exists a unique solution to the problem \eqref{eq:Fi}.

\begin{rem}\label{rem:standing}
If one considers a generic function $F$ satisfying \eqref{eq:ass-F} instead of the explicit double well potential \eqref{defF}, 
then one finds the same behaviors (i), (ii) and (iii) of Proposition \ref{prop:SW}. Indeed, the integral
\begin{equation*}
	\int_0^1 \frac{ds}{F(s)^{1/p}}
\end{equation*}
converges if and only if $n<p$ (this can be easily checked by using the assumptions \eqref{eq:ass-F}).
In particular, when $n=p$ and $F$ is given by \eqref{defF}, one has
\begin{equation*}
	\begin{aligned}
		&1-\Phi_\e(x)\approx2e^{-\lambda_px/\e}  \qquad &\mbox{as } x\to+\infty,\\ 
		&\Phi_\e(x)+1\approx2e^{\lambda_px/\e} &\mbox{as } x\to-\infty,
	\end{aligned}
\end{equation*}
where $\lambda_p:=2C_p$ and $C_p$ is defined in \eqref{eq:tanh},
while in the case of a generic potential $F$ satisfying {\eqref{eq:ass-F} with $n=p$}, one has the exponential decay 
\begin{equation*}
	\begin{aligned}
		&1-\Phi_\e(x)\leq c_1e^{-c_2x/\e}  \qquad &\mbox{as } x\to+\infty,\\ 
		&\Phi_\e(x)+1\leq c_1e^{c_2x/\e} &\mbox{as } x\to-\infty.
	\end{aligned}
\end{equation*}
On the other hand,  if $n>p$ (i.e., we are in the degenerate case),
 we have existence of a unique solution to \eqref{eq:Fi} satisfying \eqref{eq:alg-decay}.
\end{rem}

\subsubsection{Periodic solutions}
We now focus the attention on the existence of periodic solutions to \eqref{SS} and we prove the following result.
\begin{prop}\label{periodicR}
Fix $\e>0$ and let us consider problem \eqref{SS}. Then there exist periodic solutions $\Phi_{{}_{T_\e}}$ satisfying $|\Phi_{{}_{T_\e}}|<1$ and with fundamental period $T_\e$ given by
\begin{equation}\label{eq:T_eps}
	T_\e(\bar s) :={2\e}\left( \frac{p-1}{p}\right)^{1/p}\int_0^{\bar s}  \frac{ds}{\left( F(s)-F(\bar s)\right)^{1/p}}, \qquad \forall \ \bar s \in (0,1).
\end{equation}
In particular $T_\e : (0,1) \to (0,\bar T)$  if $n<p$, while  $T_\e : (0,1) \to (0,+\infty)$ if $n\geq p$.
\end{prop}

\begin{proof}
By proceeding as in the proof of Proposition \ref{prop:SW}, we want to find a solution $\Phi_{{}_{T_\e}}$ to the following problem
\begin{equation}\label{inutile}
	\frac{\e^{p}(p-1)}{p}|\Phi_{{}_{T_\e}}'|^{p}-F(\Phi_{{}_{T_\e}})=C,
\end{equation}
where $C$ in an integration constant belonging to $(-\frac{1}{2n},0)$ (since we are interested in bounded solutions). 
It is worth noticing that the choice $C=-\frac{1}{2n}$ corresponds to the constant solution $\Phi_{{}_{T_\e}}=0$, 
while the choice $C=0$ provides the heteroclinic solution found in Proposition \ref{prop:SW} (see Figure \ref{fig:1}).

\begin{figure}[hbtp]
\centering
\includegraphics[scale=.75, clip=true]{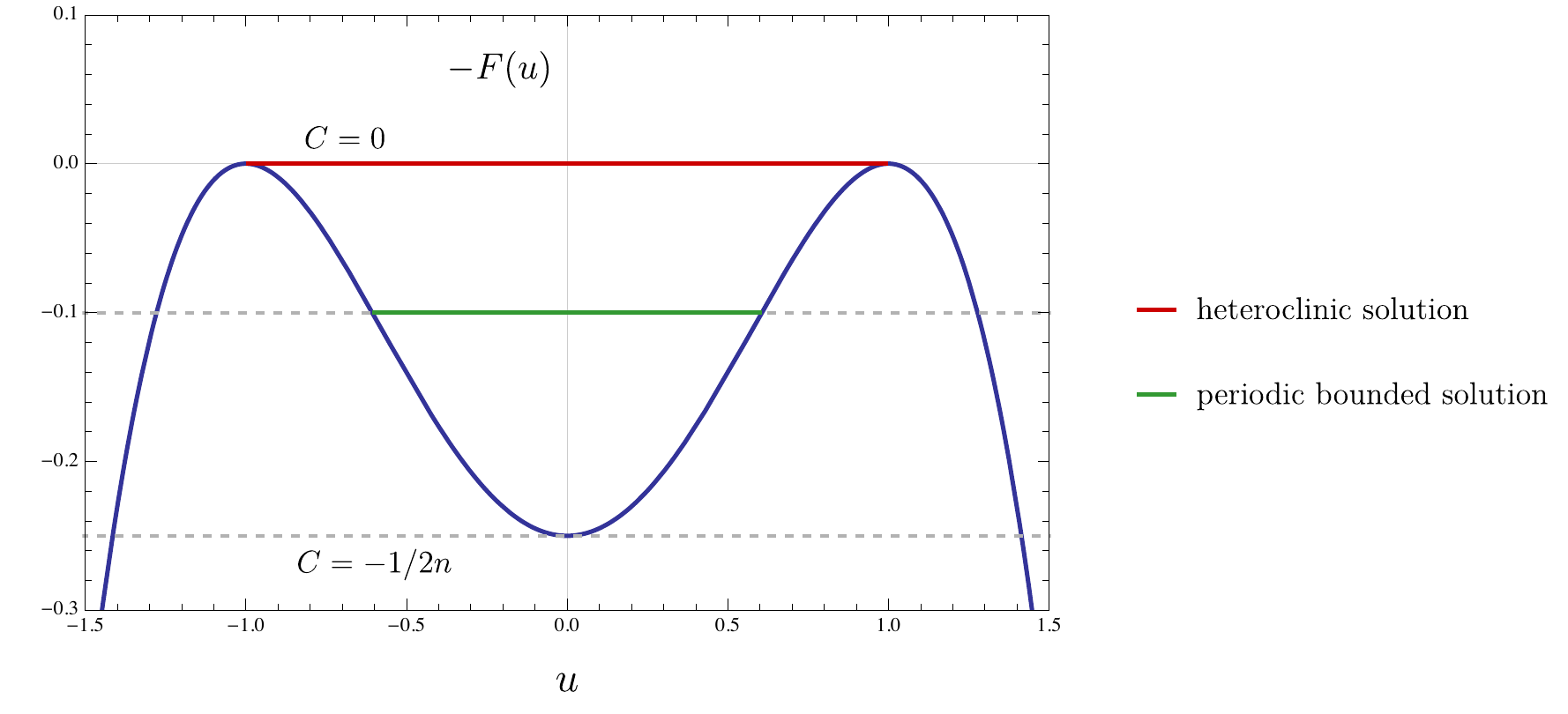}
\caption{\small{The level of the energy (given by the constant $C$) has to be such that $C> -F(u)$. In particular, bounded nontrivial solutions can be found only for $C \in (-1/2n, 0)$. When $C=0$ we have the heteroclinic solution (that touches the value $\pm 1$ only in the case $n<p$). }}
\label{fig:1}
\end{figure}

From \eqref{inutile} it follows that
\begin{equation*}
	\e|\Phi_{{}_{T_\e}}'|=A_p\displaystyle{\left(F(\Phi_{{}_{T_\e}})+C\right)^{1/p}}, \qquad  A_{p}:=  \left(\frac{p}{p-1}\right)^{1/p},
\end{equation*}
so that $\Phi_{{}_{T_\e}}$ is implicitly defined by the relation
\begin{equation*}
	\int_{0}^{\Phi_{{}_{T_\e}}(x)} \frac{ds}{\left(F(s)+C\right)^{1/p}} = A_{p}\frac{x}{\e},
\end{equation*}
where we imposed $\Phi_{{}_{T_\e}}(0)=0$ and, for definiteness, we consider $\Phi_{{}_{T_\e}}'(0)>0$. 
We thus have to verify the convergence of the following improper integral
\begin{equation*}
	\int_0^{\bar s} \frac{ds}{\left(F(s)-F(\bar s)\right)^{1/p}},
\end{equation*}
where $\bar s\in(0,1)$ is such that $F(\bar s)=-C$. 
The above integral converges whenever $p>2$ because $F'(\bar s)<0$ and we can use the Taylor expansion $F(s)-F(\bar s)=F'(\bar s)(s-\bar s)+ o(|s-\bar s|^2)$.
Thus, we have constructed a periodic solution on the whole real line with fundamental period $T_\e=T_\e(\bar s)$ given by \eqref{eq:T_eps}.
We finally notice that $T_\e(\bar s)$ is an increasing function of  $\bar s\in(0,1)$ satisfying 
\begin{equation}\label{eq:T_eps-prop}	
	\lim_{\bar s\to0^+} T_\e(\bar s)=0 \quad \mbox{ and}  \quad 
	\lim_{\bar s\to 1^-} T_\e(\bar s)=\left\{ \begin{aligned}& 
	\frac{2\e}{A_p}\int_0^{1}  \frac{ds}{F(s)^{1/p}} \qquad &\mbox{if} \ n<p,\\ &+\infty \qquad &\mbox{if} \ n\geq p, \end{aligned}\right.\, \\
\end{equation}
for any $\e>0$, and the proof is complete.
\end{proof}

\begin{rem}
We notice that, while the assumption $|\Phi_{{}_{T_\e}}|<1$ is always satisfied in the case $n \geq p$, 
when looking at the case $n<p$ such hypothesis prevents to consider solutions touching the values $\pm1$.
Indeed, in such a case the integral 
\begin{equation*}
	\int_0^1 \frac{ds}{F(s)^{1/p}}
\end{equation*}
converges, and therefore there also exist periodic solutions oscillating  between the values $\pm1$ with $|\Phi_{{}_{T_\e}}|\leq1$.
\end{rem}

\subsection{Stationary solutions in a bounded interval}\label{sub:SS}
In this subsection we consider the stationary problem associated to \eqref{eq:P-model}-\eqref{eq:Neu}, that is
\begin{equation}\label{eq:stationary}
	\e^{p}\left(|\varphi'|^{p-2}\varphi' \right)'-F'(\varphi)=0, \qquad \qquad \varphi'(a)=\varphi'(b)=0.
\end{equation}
Before proving the existence of non constant solutions to \eqref{eq:stationary}, 
we notice that all the zeros of $F'$ solve \eqref{eq:stationary};
in the case of $F$ given by \eqref{defF}, we thus have three constant solutions, $u=\pm1$ and $u=0$.

\vskip0.2cm
Our first result describes the  solutions to \eqref{eq:stationary} in the subcritical case ${1<n<p}$.
\begin{prop}\label{prop:2n<p}
Let us consider problem \eqref{eq:stationary} with ${1<n<p}$. 
Fix $N\in\mathbb{N}$ and $a<h_1<h_2<\dots<h_N<b$. If $\e>0$ is sufficiently small then there exist two solutions to \eqref{eq:stationary} oscillating between $-1$ and $+1$
with exactly $N$ zeros in $h_1,\dots,h_N$.
\end{prop}

\begin{proof}
To construct such stationary solutions, we use the standing wave introduced in Proposition \ref{prop:SW} and, 
in particular, we observe that $\Phi_\e(x)=\Psi(x/\e)$, where $\Psi$ is the monotone solution to
\begin{equation*}
	\left((\Psi')^{p-1}\right)'-F'(\Psi)=0, \quad \qquad \Psi(\pm\infty)=\pm1, \qquad \Psi(0)=0.
\end{equation*}
Proceeding as in the proof of Proposition \ref{prop:SW} (case (i)), we see that $\Psi(\pm\bar y)=\pm1$, where
\begin{equation*}
	\bar y:=\left(\frac{2n(p-1)}p\right)^{1/p}\int_{0}^1 \frac{ds}{(1-s^2)^{{n}/p}}<+\infty.
\end{equation*}
It follows that $\Phi_\e(x)=1$ if and only if $x\geq\bar x$, and $\Phi_\e(x)=-1$ if and only if $x\leq-\bar x$, where $\bar x=\e\bar y$.
Let us construct the stationary solution $\varphi_N$ with $N$ transitions in $h_1,\dots,h_N$ as follows:
assume $\varphi_N(a)=-1$ and choose $\e>0$ so small that $h_1-\e\bar y>a$ and $\e\bar y<\displaystyle\frac{h_{2}-h_1}{2}$; define
\begin{equation*}
	\varphi_N(x):=\Phi_\e(x-h_1), \qquad \qquad \mbox{ for } x\in [a,h_2-\e\bar y],
\end{equation*}
where $\Phi_\e$ is the solution to \eqref{eq:Fi}.
Observe that $\varphi_N(x)=-1$ in $[a,h_1-\e\bar y]$ and $\varphi_N(x)=1$ in $[h_1+\e\bar y,h_2-\e\bar y]$.
We have thus constructed the first transition between $-1$ and $+1$.
Let us now consider $\varphi_N$ in the interval $[h_i-\e\bar y,h_{i+1}-\e\bar y]$ for $i=2,\dots,N-1$,
where we choose $\e$ so small that $\e\bar y<\displaystyle{\min_{2\leq i\leq N-1}\frac{h_{i+1}-h_i}{2}}$.
For $i=2,\dots,N-1$, define
\begin{equation*}
	\varphi_N(x):=(-1)^{i+1}\Phi_\e(x-h_{i}), \qquad \qquad \mbox{ for } x\in [h_i-\e\bar y,h_{i+1}-\e\bar y].
\end{equation*}
Finally, by choosing $\e$  small enough such that $h_N+\e\bar y<b$, we define
\begin{equation*}
	\varphi_N(x):=(-1)^{N+1}\Phi_\e(x-h_N), \qquad \qquad \mbox{ for } x\in [h_N-\e\bar y,b].
\end{equation*}
Since $\Phi_\e$ and $-\Phi_\e$ are solutions to \eqref{eq:Fi} and \eqref{eq:Fi-}, respectively, 
the functions $\varphi_N$ and $-\varphi_N$ are thus two solutions to \eqref{eq:stationary} oscillating between $-1$ and $+1$
and with $N$ zeros in $h_1,\dots,h_N$; the proof is now complete.
\end{proof}

Notice that, because of the regularity  of $\Phi_\e$ proved in the case (i) of Proposition \ref{prop:SW}, $\varphi_N\in C^{\infty}([a,b]\backslash Z)$ for any $N$, 
where $Z=\{h_i\pm\e\bar y, \, i=1,\dots, N\}$; hence, $\pm \varphi_N$ are classical solutions to \eqref{eq:stationary}.

\vskip0.2cm
We now consider the case ${n}\geq p$. Here the integral \eqref{int_SW} diverges
and, as a consequence,  solutions cannot touch the values $\pm 1$. 
Similarly to the linear diffusion case $p=2$, we here have periodic solutions, as we state in the following Proposition.
\begin{prop}\label{prop:2n>p}
Let us consider problem \eqref{eq:stationary} with ${n\geq p>1}$ and fix $N\in \mathbb{N}$.
Then there exist two periodic {solutions} to \eqref{eq:stationary} with $N$ zeroes $h_1, \dots, h_N$ satisfying 
\begin{equation}\label{eq:h-periodic}
	h_1=a+\frac{b-a}{2N} \qquad \mbox{and} \qquad h_{i+1}=h_i+\frac{b-a}{N} \quad \mbox{for} \quad i=1,\dots,N-1.
\end{equation}
\end{prop}

\begin{proof}
We can see the periodic solutions described in the statement as the restriction on a bounded interval of periodic solutions on the whole line, 
whose existence has been proved in Proposition \ref{periodicR}.
We thus use $\Phi_{_{T_\e}}$ to construct a periodic solution $\psi_N$ with exactly $N$ zeros in the interval $[a,b]$ 
located at $h_1, \dots, h_N$ satisfying \eqref{eq:h-periodic} and with $\psi_N'(a)=\psi'_N(b)=0$ as follows:  
first, we define $\psi_N(x)=\Phi_{_{T_\e}}(x-h_1)$. 
By doing this, we are shifting the zero of $\Phi_{_{T_\e}}$ (recall that $\Phi_{_{T_\e}}(0)=0$) in the point $h_1$ and, as a consequence, 
all the other zeros of $\psi_N$ will be located {at} $h_i=h_1\pm iT_\e$, for $i=2,\dots,N$. 
Second, we impose $\psi_N'(a)=0$, which implies that $a$ has to be the middle point between $h_1-T_\e$ and $h_1$; 
consequently, we see that we have to choose $\bar s$ such that $T_\e(\bar s)=(b-a)/N$, which will thus be the period of $\psi_N$.
It is easy to check that both $\pm\psi_N$ are periodic solutions to \eqref{eq:stationary}, 
with $N$ zeros at $h_1,\dots,h_N$ satisfying \eqref{eq:h-periodic}, and the proof is complete.
\end{proof}

\begin{rem}
Proposition \ref{prop:2n>p} is valid for any $\e>0$, and throughout its proof we profited from the behavior of $T_\e(\bar s)$ only with respect to $\bar s\in(0,1)$
(see \eqref{eq:T_eps} and \eqref{eq:T_eps-prop}).
As a consequence, if $N\gg1$ (a large number of transitions), then $T_\e(\bar s)=(b-a)/N\ll1$ (a very small period) and one has to take $\bar s$ very close to $0$, 
meaning that the periodic solutions are \emph{small} oscillations around zero.
On the other hand, it is easy to check (see equation \eqref{eq:T_eps}) that $T_\e(\bar s)$ is an increasing function of $\e$, which satisfies for any $\bar s\in(0,1)$,
\begin{equation*}
	\lim_{\e\to 0^+} T_\e(\bar s)=0 \qquad \mbox{ and }  \qquad \lim_{\e\to +\infty} T_\e(\bar s)=+\infty,
\end{equation*}
and, as an alternative, one could choose an appropriate $\e$ so that $T_\e(\bar s)=(b-a)/N$.
In particular, when the interval $[a,b]$ and the number of zeros $N$ are fixed,
we have that the periodic solutions oscillate between $\pm\bar s$ with $\bar s\to1^-$ as $\e\to0^+$. 
\end{rem}

\subsection*{Comments on stationary solutions in bounded intervals} 
We underline that the main differences between Propositions \ref{prop:2n<p} and \ref{prop:2n>p} are that, in the latter, 
periodic solutions oscillate between the values $\pm 1$ without never touching them, and that the locations of the $N$ zeros is not arbitrary; 
indeed, the $N$ transitions have to be equidistant (see \eqref{eq:h-periodic}).
On the other hand, in the case ${n}< p$ the locations $h_1, \dots h_N$ are totally arbitrary. 
It is important to emphasize, however, that due to the fact that such locations are arbitrary, they can also be chosen equidistant, 
and Proposition \ref{prop:2n<p} actually includes the existence of periodic solutions as well. 
Finally it is worth noticing that, using the periodic solutions constructed in Proposition \ref{periodicR}
we can prove a result similar to Proposition \ref{prop:2n>p} also in the case $n< p$. This provides the existence of periodic solutions satisfying $|\varphi_N|<1$.

\section{The critical case ${n}=p$: exponentially slow motion}\label{sec:2n=p}
The aim of this section is to show the existence of metastable states for the initial boundary $p$-Laplacian value (IBPV) problem \eqref{eq:P-model}-\eqref{eq:Neu}-\eqref{eq:initial}, 
and that such metastable states maintain the same unstable structure of the initial datum for an exponentially long time,
i.e., for a time $T_\e\geq e^{Ap/2\e}$, with $A>0$ {dependent on $p$, but independent of $\e$}.
From now on, we shall consider equation \eqref{eq:P-model} with ${n}=p$, that is, equation
\begin{equation}\label{eq:2n=p-model}
	u_t=\e^p(|u_x|^{p-2}u_x)_x-F'(u), \qquad \mbox{ where } \qquad  F(u)= {\frac{1}{2p}|u^2-1|^p, \quad p>1.}
\end{equation}
We start by proving a lemma which shows that the energy \eqref{eq:energy} is a non-increasing function of time $t$ 
along the solutions of \eqref{eq:2n=p-model} with homogeneous Neumann boundary conditions \eqref{eq:Neu}.
\begin{lem}\label{lem:energy}
Let $u\in C([0,T],H^2(a,b))$ be solution of \eqref{eq:2n=p-model}-\eqref{eq:Neu}.
If $E_\e$ is the functional defined in \eqref{eq:energy} then 
\begin{equation}\label{eq:energy-dec}
	\frac{d}{dt}E_\e[u](t)=-\e^{-1}\|u_t(\cdot,t)\|^2_{{}_{L^2}},
\end{equation}
for any $t\in(0,T)$.
\end{lem}
\begin{proof}
By differentiating with respect to time \eqref{eq:energy}, we obtain 
\begin{equation*}
	\frac{d}{dt}E_\e[u]=\frac1\e\int_a^b\left[ \e^{p}|u_x|^{p-2} u_x u_{xt}+F'(u)u_t\right]\,dx,
\end{equation*}
Integrating by parts  and using the boundary conditions \eqref{eq:Neu} in the first term, we get
\begin{equation*}
\begin{aligned}
	\frac{d}{dt}E_\e[u]&=\frac1\e\int_a^b\left[ -\e^{p} \left(|u_x|^{p-2} u_x\right)_x u_{t}+F'(u)u_t\right]\,dx \\
	&= -\frac{1}{\e} \int_a^b u_t \left[ \e^p \left(|u_x|^{p-2} u_x\right)_x-F'(u) \right] \, dx \\
	&=-\frac{1}{\e} \int_{a}^b u_t^2 \, dx,
	\end{aligned}
\end{equation*}
where we used the fact that $u$ satisfies \eqref{eq:2n=p-model}, and we obtain exactly \eqref{eq:energy-dec}.
\end{proof}
We now make use of the generalized Young inequality
\begin{equation}\label{young}
a b \leq \frac{a^p}{p}+ \frac{b^q}{q}, \qquad \mbox{with} \quad \frac{1}{p}+\frac{1}{q}=1,
\end{equation}
with the choice 
$$a=\e |u_x| \quad \mbox{and}  \quad b= \sqrt[q]{q F(u)}=\left( \frac{p}{p-1} F(u)\right)^{\frac{p-1}{p}}.$$
We have
\begin{equation}\label{eq:c_eps}
	E_\e[{u}]\geq\int_a^b  |u_x| \left( \frac{p}{p-1} F(u)\right)^{\frac{p-1}{p}} \,dx=\left(\frac{p}{p-1}\right)^{\frac{p-1}{p}}\int_{-1}^{+1} {F(s)}^\frac{p-1}{p}\,ds=:c_p.
\end{equation}
As we will see, the positive constant $c_p$ represents the minimum energy to have a transition between $-1$ and $+1$. 
It is to be observed that when $p=2$, one has 
\begin{equation*}
	 c_2=\int_{-1}^{+1}\sqrt{2F(s)}\,ds,
\end{equation*}
which is the minimum energy in the case of the classical Allen--Cahn equation \cite{BrKo90}.
In the following, we shall improve \eqref{eq:c_eps} by showing that if a function $u$ is sufficiently close in $L^1$ to a piecewise constant function $v$
assuming only the values $\pm1$, then the energy of $u$ satisfies the \emph{lower bound},
\begin{equation*}
	E_\varepsilon[u]\geq Nc_p-C_1\exp(-C_2/\varepsilon),
\end{equation*}
for some positive constants $C_1,C_2$ independent on $\e$.
In order to prove such variational result, we give the following definitions:
let us fix here, and throughout  the rest of the paper, $N\in\mathbb{N}$ and a {\it piecewise constant function} $v$ with $N$ jumps as follows:
\begin{equation}\label{vstruct}
	v:[a,b]\rightarrow\{-1,1\}\  \hbox{with $N$ jumps located at } a<h_1<h_2<\cdots<h_N<b.
\end{equation}	
Moreover, we fix $r>0$ such that
\begin{equation}\label{eq:r}
	h_i+r<h_{i+1}-r, \ \hbox{ for}\ i=1,\dots,N, \qquad   a\leq h_1-r,\qquad h_N+r\leq b.
\end{equation}
Finally, for {$p>1$}, define 
\begin{equation}\label{eq:lambda}
	\lambda_p:= 2^{1-\frac{1}{p}}{(p-1)^{-\frac{1}{p}}}.
\end{equation}
\begin{prop}\label{prop:lower}
Let $F$ be as in \eqref{eq:2n=p-model}, $v$ be as in \eqref{vstruct}, and fix $A\in(0,r\sqrt2\lambda_p)$, 
where $r$ satisfies \eqref{eq:r} and $\lambda_p$ is defined in \eqref{eq:lambda}.
Then there exist $\e_0,C,\delta>0$ (depending only on $p,v$ and $A$) such that if $u\in H^1(a,b)$ satisfies 
\begin{equation}\label{eq:u-v}
	\|u-v\|_{{}_{L^1}}\leq\delta,
\end{equation}
then for any $\e\in(0,\e_0)$,
\begin{equation}\label{eq:lower}
	E_\varepsilon[u]\geq Nc_p-C\exp(-Ap/2\varepsilon),
\end{equation}
where $E_\e$ and $c_p$ are defined in \eqref{eq:energy} and \eqref{eq:c_eps}, respectively.
\end{prop}
\begin{proof}
Fix $u\in H^1(a,b)$ satisfying \eqref{eq:u-v}, and take $\hat r\in(0,r)$ and $\rho_1$ so small that 
\begin{equation}\label{eq:nu}
	{A\leq\lambda_p(r-\hat r)\sqrt{2-3\rho_1}}.
\end{equation}
Then, choose $0<\rho_2 < \rho_1$  sufficiently small such that
\begin{equation}\label{eq:forrho2}
\begin{aligned}
	\int_{1-\rho_1}^{1-\rho_2}F(s)^{\frac{p-1}{p}}\,ds&>\int_{1-\rho_2}^{1}F(s)^{\frac{p-1}{p}}\,ds,  \\
	\int_{-1+\rho_2}^{-1+\rho_1}F(s)^{\frac{p-1}{p}}\,ds&> \int_{-1}^{-1+\rho_2}F(s)^{\frac{p-1}{p}}\,ds.
	\end{aligned}
\end{equation}
We now focus our attention on $h_i$, one of the points of discontinuity of $v$. To fix ideas, 
let $v(h_i\pm r)=\pm1$, the other case being analogous.
We claim that there exist $r_+$ and $r_-$ in $(0,\hat r)$ such that
\begin{equation}\label{2points}
	|u(h_i+r_+)-1|<\rho_2, \qquad \quad \mbox{ and } \qquad \quad |u(h_i-r_-)+1|<\rho_2.
\end{equation}
Indeed, assume by contradiction that $|u-1|\geq\rho_2$ for $u \in (h_i,h_i+\hat r)$; then
\begin{equation*}
	\delta\geq\|u-v\|_{{}_{L^1}}\geq\int_{h_i}^{h_i+\hat r}|u-v|\,dx\geq\hat r\rho_2,
\end{equation*}
and this leads to a contradiction if we choose $\delta\in(0,\hat r\rho_2)$.
Similarly, one can prove the existence of $r_-\in(0,\hat r)$ such that $|u(h_i-r_-)+1|<\rho_2$.

Now, we consider the interval $(h_i-r,h_i+r)$ and claim that
\begin{equation}\label{eq:claim}
	\int_{h_i-r}^{h_i+r}\left[\frac{\e^{p-1}|u_x|^p}{p}+\frac{F(u)}\e\right]\,dx\geq c_p-\frac{C}N\exp(-Ap/2\varepsilon),
\end{equation}
for some $C>0$ independent on $\e$.
Observe that from \eqref{eq:c_eps}, it follows that for any $a\leq c<d\leq b$,
\begin{equation}\label{eq:ineq}
	\int_c^d\left[\frac{\e^{p-1}|u_x|^p}{p}+\frac{F(u)}\e\right]\,dx \geq \left|\int_{u(c)}^{u(d)}\left( \frac{p}{p-1} F(s)\right)^{\frac{p-1}{p}} \,ds\right|.
\end{equation}
Hence, if $u(h_i+r_+)\geq1$ and $u(h_i-r_-)\leq-1$, then from \eqref{eq:ineq} we can conclude that
\begin{equation*}
	\int_{h_i-r_-}^{h_i+r_+}\left[\frac{\e^{p-1}|u_x|^p}{p}+\frac{F(u)}\e\right]\,dx\geq c_p,
\end{equation*}
which implies \eqref{eq:claim}. On the other hand, from \eqref{eq:ineq} we obtain
\[
\int_{h_i-r_-}^{h_i+r_+}\left[\frac{\e^{p-1}|u_x|^p}{p}+\frac{F(u)}\e\right]\,dx \geq \int_{u(h_i - r_-)}^{u(h_i + r_+)} \left( \frac{p}{p-1} F(s)\right)^{\frac{p-1}{p}}\,ds,
\]
yielding, in turn,
%
%notice that in general we have
\begin{align}
	\int_{h_i-r}^{h_i+r}\left[\frac{\e^{p-1}|u_x|^p}{p}+\frac{F(u)}\e\right]\,dx & 
	\geq \int_{h_i+r_+}^{h_i+r}\left[\frac{\e^{p-1}|u_x|^p}{p}+\frac{F(u)}\e\right]\,dx\notag \\ 
	& \quad + \int_{h_i-r}^{h_i-r_-}\left[\frac{\e^{p-1}|u_x|^p}{p}+\frac{F(u)}\e\right]\,dx \notag \\
	& \quad +\int_{-1}^{1}\left( \frac{p}{p-1} F(s)\right)^{\frac{p-1}{p}}\,ds\notag\\
	&\quad-\int_{-1}^{u(h_i-r_-)}\left( \frac{p}{p-1} F(s)\right)^{\frac{p-1}{p}}\,ds \notag \\
	& \quad-\int_{u(h_i+r_+)}^{1}\left( \frac{p}{p-1} F(s)\right)^{\frac{p-1}{p}}\,ds\notag \\
	&=:I_1+I_2+c_p-\alpha_p-\beta_p. \label{eq:Pe}
\end{align}
%where we again used \eqref{eq:ineq}.
Let us estimate the first two terms of \eqref{eq:Pe}. 
Regarding $I_1$, assume that $1-\rho_2<u(h_i+r_+)<1$ and consider the unique minimizer $z:[h_i+r_+,h_i+r]\rightarrow\R$ 
of $I_1$ subject to the boundary condition $z(h_i+r_+)=u(h_i+r_+)$.
If the range of $z$ is not contained in the interval $(1-\rho_1,1+\rho_1)$, then from \eqref{eq:ineq} it follows that
\begin{equation}\label{E>fi}
	\int_{h_i+r_+}^{h_i+r}\left[\frac{\e^{p-1}|u_x|^p}{p}+\frac{F(u)}\e\right]\,dx \geq \int_{u(h_i+r_+)}^{1}\left( \frac{p}{p-1} F(s)\right)^{\frac{p-1}{p}}\,ds=\beta_p,
\end{equation}
by the choice of $r_+$ and $\rho_2$. 
Suppose, on the other hand, that the range of $z$ is contained in the interval $(1-\rho_1,1+\rho_1)$. 
Then, the Euler-Lagrange equation for $z$ is
\begin{equation}\label{Euler}
\begin{aligned}
	&z''(x) z'(x)^{p-2}=\e^{-p}(p-1)^{-1}F'(z(x)), \quad \qquad x\in(h_i+r_+,h_i+r),\\
	&z(h_i+r_+)=u(h_i+r_+), \quad \qquad z'(h_i+r)=0.
\end{aligned}
\end{equation}
For later use, multiply \eqref{Euler} by $z'(x)$ to obtain
$$
\left( \frac{\e^p}{p}(p-1)z'(x)^p-F(z(x))\right)'=0,
$$
which implies
$$
\frac{\e^p}{p}(p-1)z'(x)^p= F(z(x))-F(z(h_i+r)).
$$
In particular
\begin{equation}\label{uselater}
\frac{\e^p}{p}(p-1)z'(x)^p< F(z(x)).
\end{equation}
Let us now define $\psi(x):=(z(x)-1)^2$. Then we have $\psi'=2(z-1)z'$ and 
\begin{equation*}
\begin{aligned}
	\psi''(x)&=2(z(x)-1)z''(x)+2z'(x)^2 \\
	&\geq2(z(x)-1)z''(x) \\
	& = 2(z(x)-1)\frac{F'(z(x))}{\e^p(p-1)}z'(x)^{2-p} \\
	& \geq \frac{2}{\e^{2-p}}  \left( \frac{p}{p-1}\right)^{{\frac{2-p}{p}}}(z(x)-1)\frac{F'(z(x))}{\e^p(p-1)}  F(z(x))^{\frac{2-p}{p}},
\end{aligned}
\end{equation*}
where in the last inequality we used \eqref{uselater} and the fact that $F'(z)(z-1) \geq 0$ for $z \in (1-\rho_1,1+\rho_1)$. Recalling the explicit form of $F$, we end up with
\begin{equation*}
\begin{aligned}
\psi''(x) & \geq \frac{2}{\e^2} \frac{1}{(p-1)^{\frac{2}{p}}} \frac{1}{2^{\frac{2-p}{p}}} (z(x)-1)^2(z(x)+1) z(x) \\
&\geq \frac{2}{\e^2} \frac{1}{(p-1)^{\frac{2}{p}}} \frac{1}{2^{\frac{2-p}{p}}} (z(x)-1)^2 (2-\rho_1)(1-\rho_1)\\
&\geq \frac{2^{2-\frac2p}}{\e^2}\frac{1}{(p-1)^{\frac{2}{p}}}(2-3\rho_1) \psi(x) \\
& = \frac{\lambda^2}{\e^2}(2-3\rho_1)\psi(x),
\end{aligned}
\end{equation*}
where we used the fact that $|z(x)-1|\leq\rho_1$ for any $x\in[h_i+r_+,h_i+r]$.
Denoting by $\mu=A/(r-\hat r)$ and using \eqref{eq:nu} we get
\begin{align*}
	\psi''(x)-\frac{\mu^2}{\varepsilon^2}\psi(x)\geq0, \quad \qquad x\in(h_i+r_+,h_i+r),\\
	\psi(h_i+r_+)=(u(h_i+r_+)-1)^2, \quad \qquad \psi'(h_i+r)=0.
\end{align*}
It follows that we can compare $\psi$ with the solution $\hat \psi$ to
\begin{align*}
	\hat\psi''(x)-\frac{\mu^2}{\varepsilon^2}\hat\psi(x)=0, \quad \qquad x\in(h_i+r_+,h_i+r),\\
	\hat\psi(h_i+r_+)=(u(h_i+r_+)-1)^2, \quad \qquad \hat\psi'(h_i+r)=0,
\end{align*}
which can be explicitly calculated to be
\begin{equation*}
	\hat\psi(x)=\frac{(u(h_i+r_+)-1)^2}{\cosh\left[\frac\mu\varepsilon(r-r_+)\right]}\cosh\left[\frac\mu\varepsilon(x-(h_i+r))\right].
\end{equation*}
By the maximum principle $\psi(x)\leq\hat\psi(x)$, implying
\begin{equation*}
	\psi(h_i+r)\leq\frac{(u(h_i+r_+)-1)^2}{\cosh\left[\frac\mu\varepsilon(r-r_+)\right]}\leq2(u(h_i+r_+)-1)^2 \exp\left(-\frac{\mu (r-\hat r)}{\e}\right).
\end{equation*}
Since $A=\mu (r-\hat r)$, we thus have 
\begin{equation}\label{|z-v+|<exp}
	|z(h_i+r)-1|\leq\sqrt2\rho_2\exp(-A/2\varepsilon).
\end{equation}
Finally, recalling that $F(s)= {\frac{1}{2p}|1-s^2|^p}$ and using \eqref{|z-v+|<exp} we obtain
\begin{equation}\label{fi<exp}
\begin{aligned}
	\left|\int_{z(h_i+r)}^{1}\left( \frac{p}{p-1} F(s)\right)^{\frac{p-1}{p}}\,ds\right|&\leq \frac{1}{p}\left( \frac{1}{2(p-1)}\right)^{\frac{p-1}{p}}|z(h_i+r)-1|^{{p}} \\
	&\leq \frac{2^{\frac{p}{2}}}{p} \left( \frac{1}{2(p-1)}\right)^{\frac{p-1}{p}}\,\rho_2^{p}\,\exp(-Ap/2\varepsilon) \\
	& =K_p\exp(-Ap/2\varepsilon).
	\end{aligned}
\end{equation}
From \eqref{eq:ineq}-\eqref{fi<exp} it follows that, 
\begin{align}
	\int_{h_i+r_+}^{h_i+r}\left[\frac{\e^{p-1}|z_x|^p}{p}+\frac{F(z)}\e\right]\,dx &\geq \left|\int_{z(h_i+r_+)}^{1}\left( \frac{p}{p-1} F(s)\right)^{\frac{p-1}{p}}\,ds\,-\right.\nonumber \\
	&\qquad \qquad\left.\int_{z(h_i+r)}^{1}\left( \frac{p}{p-1} F(s)\right)^{\frac{p-1}{p}}\,ds\right| \nonumber\\
	& \geq\beta_p-\frac{C}{2N}\exp(-Ap/2\varepsilon), \label{E>fi-exp}
\end{align}
where
$C=2N K_p > 0.$
Combining \eqref{E>fi} and \eqref{E>fi-exp}, we get that, independently  on its range, the  minimizer $z$ of the proposed variational problem satisfies 
\begin{equation*}	
	\int_{h_i+r_+}^{h_i+r}\left[\frac{\e^{p-1}|z_x|^p}{p}+\frac{F(z)}\e\right]\,dx \geq\beta_\e-\frac{C}{2N}\exp(-Ap/2\varepsilon).
\end{equation*}
The restriction of $u$ to $[h_i+r_+,h_i+r]$ is an admissible function. Hence, it satisfies the same estimate and we have
\begin{equation}\label{eq:I1}
	I_1\geq\beta_p-\frac{C}{2N}\exp(-Ap/2\varepsilon).
\end{equation}
The term $I_2$ on the right hand side of \eqref{eq:Pe} is estimated similarly by analyzing  the interval $[h_i-r,h_i-r_-]$ 
and using the second condition of \eqref{eq:forrho2} to obtain the corresponding inequality \eqref{E>fi}.
The obtained lower bound reads
\begin{equation}\label{eq:I2}	
	I_2\geq\alpha_p-\frac{C}{2N}\exp(-Ap/2\varepsilon).
\end{equation}
Finally, by substituting \eqref{eq:I1} and \eqref{eq:I2} in \eqref{eq:Pe}, we deduce \eqref{eq:claim}.
Summing up all of these estimates for $i=1, \dots, N$, namely for all transition points, we end up with
\begin{equation*}
	E_\varepsilon[u]\geq\sum_{i=1}^N\int_{h_i-r}^{h_i+r}\left[\frac{\e^{p-1}|u_x|^p}{p}+\frac{F(u)}\e\right]\,dx\geq Nc_p-C\exp(-Ap/2\varepsilon),
\end{equation*}
and the proof is complete.
\end{proof}
Lemma \ref{lem:energy} and Proposition \ref{prop:lower} are the key ingredients to apply the energy approach 
introduced in \cite{BrKo90} and to proceed as in \cite{Fol19,FLM19,FPS1,Grnt95}.
First of all, we give the definition of a function with a \emph{transition layer structure}. 
\begin{defn}\label{def:TLS}
We say that a function $u^\e\in H^1(a,b)$ has an \emph{$N$-transition layer structure} if 
\begin{equation}\label{eq:ass-u0}
	\lim_{\varepsilon\rightarrow 0} \|u^\varepsilon-v\|_{{}_{L^1}}=0,
\end{equation}
where $v$ is as in \eqref{vstruct}, and there exist constants $C>0$, $A\in(0,r\sqrt2\lambda_p)$ 
(with $r$ satisfying \eqref{eq:r} and $\lambda_p$ defined in \eqref{eq:lambda}) such that
\begin{equation}\label{eq:energy-ini}
	E_\varepsilon[u^\varepsilon]\leq Nc_p+C\exp(-Ap/2\e),
\end{equation}
for any $\varepsilon\ll1$, where the energy $E_\e$ and the positive constant $c_ p$ are defined in \eqref{eq:energy} and in \eqref{eq:c_eps}, respectively.
\end{defn}
\begin{rem}
We stress that functions satisfying Definition \ref{def:TLS} are neither stationary solutions to \eqref{eq:2n=p-model} nor they are close to them; 
indeed, in the case $n = p$ (and in the case $n > p$ as well) the only non constant stationary solutions are periodic (cfr. Proposition \ref{prop:2n>p}), 
while the transition points $h_j$  of Definition \ref{def:TLS} are chosen arbitrarily. 
In contrast, when ${n}< p$ stationary solutions constructed in Proposition \ref{prop:2n<p} satisfy conditions \eqref{eq:ass-u0}-\eqref{eq:energy-ini} 
(see Proposition \ref{prop:ex-met} and Remark \ref{rem:translayer}); 
hence, as we have already remarked in the Introduction, they are invariant under the dynamics of \eqref{eq:P-model} and there is no interest in studying their long time behavior.

Finally, observe that condition \eqref{eq:ass-u0} fixes the number of transitions and their relative positions in the limit $\varepsilon\to0$,
while condition \eqref{eq:energy-ini} requires that the energy  exceeds at most by $C\exp(-Ap/2\e)$, the minimum possible to have these $N$ transitions. 
Moreover, from \eqref{eq:energy-dec} it follows that if the initial datum $u_0^\e$ satisfies \eqref{eq:energy-ini}, 
then the solution $u^\e(\cdot,t)$ satisfies the same inequality for all times $t>0$.
\end{rem}

The main result of this section states that the solution $u^\e(\cdot,t)$ arising from an initial datum satisfying \eqref{eq:ass-u0} and \eqref{eq:energy-ini}, 
satisfies the property \eqref{eq:ass-u0}  as well, for an exponentially long time.

\begin{thm}[metastable dynamics with $p$-Laplacian diffusion in the critical case ${n}=p$]\label{thm:main}
Let $v$ be as in \eqref{vstruct} and $A\in(0,r\sqrt{2}\lambda_p)$, with $r$ satisfying \eqref{eq:r} and $\lambda_p$ defined in \eqref{eq:lambda}.
If $u^\varepsilon$ is the solution to \eqref{eq:2n=p-model}
with homogeneous Neumann boundary conditions \eqref{eq:Neu} and initial datum $u_0^{\varepsilon}$
satisfying \eqref{eq:ass-u0} and \eqref{eq:energy-ini}, then, 
\begin{equation}\label{eq:limit}
	\sup_{0\leq t\leq m \, {\exp(Ap/2\varepsilon)}}\|u^\varepsilon(\cdot,t)-v\|_{{}_{L^1}}\xrightarrow[\varepsilon\rightarrow0]{}0,
\end{equation}
for any $m>0$.
\end{thm}
As it was already mentioned, thanks to Lemma \ref{lem:energy} and Proposition \ref{prop:lower}, 
we can apply the same strategy of \cite{BrKo90} to prove Theorem \ref{thm:main}.
The first step of the proof is the following bound on the $L^2$--norm of the time derivative of the solution $u_t^\varepsilon$.
\begin{prop}\label{prop:L2-norm}
Let us consider the solution $u^\e$ to \eqref{eq:2n=p-model}
with homogeneous Neumann boundary conditions \eqref{eq:Neu} and initial datum 
$u_0^{\varepsilon}$ 
%\eqref{eq:initial}, 
which satisfies \eqref{eq:ass-u0} and \eqref{eq:energy-ini}.
Then there exist positive constants $\varepsilon_0, C_1, C_2>0$ (independent on $\varepsilon$) such that
\begin{equation}\label{L2-norm}
	\int_0^{C_1\varepsilon^{-1}\exp(Ap/2\varepsilon)}\|u_t^\varepsilon\|^2_{{}_{L^2}}dt\leq C_2\varepsilon\exp(-Ap/2\varepsilon),
\end{equation}
for all $\varepsilon\in(0,\varepsilon_0)$.
\end{prop}

\begin{proof}
Let $\varepsilon_0>0$ be sufficiently small such that, for all $\varepsilon\in(0,\varepsilon_0)$, \eqref{eq:energy-ini} holds and 
\begin{equation}\label{1/2delta}
	\|u_0^\varepsilon-v\|_{{}_{L^1}}\leq\frac12\delta,
\end{equation}
where $\delta$ is the constant of Proposition \ref{prop:lower}. 
Let $\hat T>0$; we claim that if
\begin{equation}\label{claim1}
	\int_0^{\hat T}\|u_t^\varepsilon\|_{{}_{L^1}}dt\leq\frac12\delta,
\end{equation}
then there exists $C>0$ such that
\begin{equation}\label{claim2}
	E_\varepsilon[u^\varepsilon](\hat T)\geq Nc_p-C\exp(-Ap/2\varepsilon).
\end{equation}
Indeed, by using \eqref{1/2delta}, \eqref{claim1} and the triangle inequality we obtain
\begin{equation*}
	\|u^\varepsilon(\cdot,\hat T)-v\|_{{}_{L^1}}\leq\|u^\varepsilon(\cdot,\hat T)-u_0^\varepsilon\|_{{}_{L^1}}+\|u_0^\varepsilon-v\|_{{}_{L^1}}
	\leq\int_0^{\hat T}\|u_t^\varepsilon\|_{{}_{L^1}}+\frac12\delta\leq\delta,
\end{equation*}
and inequality \eqref{claim2} follows from Proposition \ref{prop:lower}.
Upon integration of \eqref{eq:energy-dec}, we deduce
\begin{equation}\label{dissipative}
	E_\e[u^\e_0]-E_\e[u^\e](\hat T)=\e^{-1}\int_0^{\hat T}\|u_t^\e\|^2_{{}_{L^2}}\,dt.
\end{equation}
Substituting  \eqref{eq:energy-ini} and \eqref{claim2} in \eqref{dissipative} yields
\begin{equation}\label{L2-norm-Teps}
	\int_0^{\hat T}\|u_t^\varepsilon\|^2_{{}_{L^2}}dt\leq C_2\e\exp(-Ap/2\varepsilon).
\end{equation}
It remains to prove that inequality \eqref{claim1} holds for $\hat T\geq C_1\e^{-1}\exp(Ap/2\varepsilon)$.
If 
\begin{equation*}
	\int_0^{+\infty}\|u_t^\varepsilon\|_{{}_{L^1}}dt\leq\frac12\delta,
\end{equation*}
then there is nothing to prove. 
Otherwise, choose $\hat T$ such that
\begin{equation*}
	\int_0^{\hat T}\|u_t^\varepsilon\|_{{}_{L^1}}dt=\frac12\delta.
\end{equation*}
Using H\"older's inequality and \eqref{L2-norm-Teps}, we infer
\begin{equation*}
	\frac12\delta\leq[\hat T(b-a)]^{1/2}\biggl(\int_0^{\hat T}\|u_t^\varepsilon\|^2_{{}_{L^2}}dt\biggr)^{1/2}\leq
	\left[\hat T(b-a)C_2\varepsilon\exp(-Ap/2\varepsilon)\right]^{1/2}.
\end{equation*}
It follows that there exists $C_1>0$ such that
\begin{equation*}
	\hat T\geq C_1\varepsilon^{-1}\exp(Ap/2\varepsilon),
\end{equation*}
and the proof is complete.
\end{proof}
We now have all the tools to prove \eqref{eq:limit}.
\begin{proof}[Proof of Theorem \ref{thm:main}]
The triangle inequality yields
\begin{equation}\label{trianglebar}
	\|u^\varepsilon(\cdot,t)-v\|_{{}_{L^1}}\leq\|u^\varepsilon(\cdot,t)-u_0^\varepsilon\|_{{}_{L^1}}+\|u_0^\varepsilon-v\|_{{}_{L^1}},
\end{equation}
for all $t\in[0,m\exp(Ap/2\varepsilon)]$. 
The last term of inequality \eqref{trianglebar} tends to zero by assumption \eqref{eq:ass-u0}.
Regarding the first term, take $\varepsilon$ so small that $C_1\varepsilon^{-1}\geq m$;
 we can thus apply Proposition \ref{prop:L2-norm}, and by  H\"older's inequality and \eqref{L2-norm}, we infer
\begin{equation*}
	\sup_{0\leq t\leq m\exp(Ap/2\varepsilon)}\|u^\e(\cdot,t)-u^\e_0\|_{{}_{L^1}}\leq\int_0^{m\exp(Ap/2\varepsilon)}\|u_t^\e(\cdot,t)\|_{{}_{L^1}}\,dt\leq C\sqrt\e,
\end{equation*}		
for all $t\in[0,m\exp(Ap/2\varepsilon)]$. Hence \eqref{eq:limit} follows.
\end{proof}

Theorem \ref{thm:main} provides sufficient conditions for the existence of a metastable state for equation \eqref{eq:2n=p-model}
and shows its persistence for (at least) an exponentially long time.
Also, we recover exactly the classical result when $p=2$ (cfr. \cite{CaPe89}).

\subsection{Construction of a function with a $N-$transition layer structure }
We conclude this section by constructing a family of functions with a transition layer structure
satisfying \eqref{eq:ass-u0}-\eqref{eq:energy-ini}.
To do this, we will use the standing waves solutions introduced in Section \ref{substanding},
\begin{equation}\label{eq:Fi2}
	\e^p\left(|\Phi'_\e|^{p-2}\Phi'_\e\right)'-F'(\Phi_\e)=0, \quad \qquad \lim_{x\to\pm\infty}\Phi_\e(x)=\pm1, \qquad \Phi_\e(0)=0,
\end{equation}
whose existence has been proved in Proposition \ref{prop:SW}.
More precisely, we prove the following result.
\begin{prop}\label{prop:ex-met}
Fix a piecewise function $v$ as in \eqref{vstruct} and assume that $F(u)=\frac1{2p}{|1-u^2|}^p$ with $p{>1}$.
Then there exists a function $u^\e$ satisfying \eqref{eq:ass-u0} and 
\begin{equation}\label{eq:Nc_eps}
	E_\e[u^\e]<Nc_p,
\end{equation}
where $E_\e$ and $c_p$ are defined in \eqref{eq:energy} and \eqref{eq:c_eps}, respectively.
\end{prop}
\begin{proof}
In order to construct a family of functions with a transition layer structure, we use the solution $\Phi_\e$ to \eqref{eq:Fi2}.
Fix $N\in\mathbb{N}$ and $N$ transition points $a<h_1<h_2<\dots<h_n<b$, and denote the middle points by
\begin{equation*}
	m_1:=a, \qquad \quad m_j:=\frac{h_{j-1}+h_j}{2}, \quad j=2,\dots,N-1, \qquad \quad m_N:=b.
\end{equation*}
Define
\begin{equation}\label{eq:translayer}
	u^\e(x):=(-1)^j\Phi_\e\left(x-h_j\right), \qquad \qquad x\in[m_j,m_{j+1}], \qquad \qquad j=1,\dots N,
\end{equation}
where $\Phi_\e$ and $-\Phi_\e$ are the solutions to \eqref{eq:Fi2} and \eqref{eq:Fi-}, respectively.
Notice that $u^\e(h_j)=0$, for $j=1,\dots,N$ and for definiteness we choose $u^\e(a)<0$ (the case $u^\e(a)>0$ is analogous).
Let us prove that $u^\e$ has an $N$-transition layer structure, i.e., that it satisfies \eqref{eq:ass-u0}-\eqref{eq:energy-ini}.
It is easy to check that $u^\e\in H^1(a,b)$ and satisfies \eqref{eq:ass-u0}; let us show that \eqref{eq:Nc_eps} holds true.

From the definitions of $E_\e$ and $u^\e$ we obtain
\begin{equation*}
	E_\e[u^\e]=\sum_{j=1}^{N}\int_{m_j}^{m_{j+1}}\left[\frac{\e^{p-1}|\Phi'_\e|^p}{p}+\frac{F(\Phi_\e)}{\e}\right]\,dx.
\end{equation*}
From \eqref{eq:FI-first}, it follows that $\e^p|\Phi'_\e|^p=\frac{p}{p-1}F(\Phi_\e)$ and so
\begin{equation*}
	\int_{m_j}^{m_{j+1}}\left[\frac{\e^{p-1}|\Phi'_\e|^p}{p}+\frac{F(\Phi_\e)}{\e}\right] \,dx=
	\int_{m_j}^{m_{j+1}}|\Phi'_\e|\left(\frac{p}{p-1}F(\Phi_\e)\right)^{\frac{p-1}{p}} \, dx <c_p,
\end{equation*}
where $c_p$ is defined in \eqref{eq:c_eps}.
Summing up all the terms we end up with \eqref{eq:Nc_eps} and the proof is complete.
\end{proof}

\begin{rem}\label{rem:translayer}
The previous result can be easily extended to a generic potential of the form \eqref{defF}, that is for generic $n$ and $p$.
Indeed, in the proof of Proposition \ref{prop:ex-met} we constructed the function $u^\e$ by using the standing wave  $\Phi_\e$ solution to \eqref{eq:Fi2}
and we only used the fact that $\Phi_\e$ satisfies \eqref{eq:FI-first} and $|\Phi_\e|<1$.
Hence, by making use of Proposition \ref{prop:SW}, one can easily extend the results of Proposition \ref{prop:ex-met} to the case ${n}\neq p$;
in particular, the property \eqref{eq:ass-u0} is always satisfied, while \eqref{eq:Nc_eps} is only valid in the case $n>p$.
Indeed, when $n<p$ one has $E_\e[u^\e]=Nc_p$ for $\e$ sufficiently small
since the profile $\Phi_\e$ touches $\pm1$ (see Proposition \ref{prop:ex-met}).
\end{rem}

\section{The supercritical or degenerate case ${n}>p$: algebraic slow motion}\label{sec:degcase}
In this section we show that,  in the degenerate case with $n>p$, solutions to \eqref{eq:P-model}-\eqref{eq:Neu}-\eqref{eq:initial}
with a transition layer structure will maintain this shape for a time of $\mathcal O(\e^{-k})$, for some $k>0$. 
Hence, here the order of the convergence of unstable solutions towards the minima of the potential $F$ is only algebraic in $\e$, and no metastability is observed, in contrast to the critical case $n = p$ studied in Section \ref{sec:2n=p}. 
The same behavior has been previously observed for reaction diffusion problems ($p=2$) in \cite{BetSme2013},
and it is a direct consequence of the fact that, since $n>p$ and
\begin{equation}\label{defF2}
	F(u)={\frac{1}{2n}|1-u^2|^{n}, \quad n>1,}
\end{equation}
then $F^{([p])}(\pm 1)=0$, that is, we are in the {\it degenerate case}.

The first result we present is the analogous of Proposition \ref{prop:lower}, and provides a lower bound on the energy \eqref{eq:energy}.
%We underline that such a result is purely variational in character and equation \eqref{eq:P-model} plays no role in its proof.
\begin{prop}\label{prop:lower_deg}
Let $p{>1}$, $F$ given by \eqref{defF2} with $n>p$,
$v:(a,b)\rightarrow\{-1,+1\}$ a piecewise constant function with exactly $N$ discontinuities (as in \eqref{vstruct}) and define the sequence
\begin{equation}\label{eq:exp_alg}
	\begin{cases}
		k_1=0,\\
		k_2:=\alpha, \\
         	k_{m+1}:=\alpha(k_{m}+1), \qquad m\geq2, 
	\end{cases}
	\quad \mbox{where}  \quad \alpha:=\displaystyle\frac{p-1}{p}+\frac{1}{n}.
\end{equation}
Then, for any $m\in \mathbb N$ there exist constants $\delta_m>0$ and $C_m>0$ such that if $w\in H^1$ satisfies
\begin{equation}\label{|w-v|_L^1<delta_l}
	\|w-v\|_{L^1}\leq\delta_m,
\end{equation}
and
\begin{equation}\label{E_p(w)<Nc0+eps^l}
	E_\e[w]\leq Nc_p+\varepsilon^{k_{m}},
\end{equation}
with $\varepsilon$ sufficiently small, then
\begin{equation}\label{E_p(w)>Nc0-eps^l}
	E_\e[w]\geq Nc_p-C_m\varepsilon^{k_{m+1}},
\end{equation}
where $E_\e$ and $c_p$ are defined in \eqref{eq:energy} and \eqref{eq:c_eps}, respectively.
\end{prop}

\begin{proof}
First of all, we observe that the assumption $n>p$ implies that $\alpha\in(0,1)$ and, consequently, 
the increasing sequence defined in \eqref{eq:exp_alg} satisfies
\begin{equation*}
	\lim_{m\to+\infty}k_m=\frac{\alpha}{1-\alpha}=\frac{np}{n-p}-1.
\end{equation*}
We prove our statement by induction on $m\geq1$ and we begin our proof by considering the case of only one transition ($N=1$). Let $h_1$ be the only point of discontinuity of $v$ and assume, without loss of generality, that $v=-1$ on $(a,h_1)$. 
Also, we choose $\delta_m$ small enough such that
\begin{equation*}
	(h_1-2m\delta_m,h_1+2m\delta_m)\subset(a,b).
\end{equation*}
Our goal is to show that for any $m\in\mathbb N$ there exist $x_m\in(h_1- 2m\delta_m,h_1)$ and $y_m\in(h_1, h_{1}+2m\delta_m)$ such that
\begin{equation}\label{x_k,y_k}
	w(x_m)\leq-1+C_m\varepsilon^\frac{k_m+1}{n}, \qquad w(y_m)\geq 1-C_m\varepsilon^\frac{k_m+1}{n},
\end{equation}
and
\begin{equation}\label{E_p(w)-xk,yk}
	\int_{x_m}^{y_m}\left[\frac{\e^{p-1}}p|w_x|^p+\frac{F(w)}\varepsilon\right]dx\geq c_p-C_m\varepsilon^{k_{m+1}},
\end{equation}
where $\{k_m\}_{{}_{m\geq1}}$ is defined in \eqref{eq:exp_alg}.
We start with the base case $m=1$, and we show that hypotheses \eqref{|w-v|_L^1<delta_l} and \eqref{E_p(w)<Nc0+eps^l} imply the existence of two points 
$x_1{\in(h_1-2\delta_1,h_1)}$ and $y_1\in(h_1,h_1+2\delta_1)$ such that
\begin{equation}\label{x_1,y_1}
	w(x_1)\leq-1+C_1\varepsilon^\frac1{n}, \qquad w(y_1)\geq1-C_1\varepsilon^\frac1{n}\,.
\end{equation}
From hypothesis \eqref{|w-v|_L^1<delta_l} in the case $m=1$, we have
\begin{equation}\label{int su (gamma,b)}
	\int_{h_1}^b|w-1|\leq\delta_1,
\end{equation}
so that, denoting by $S^-:=\{y:w(y)\leq0\}$ and by $S^+:=\{y:w(y)>0\}$, 
 \eqref{int su (gamma,b)} yields
\begin{equation*}
\begin{aligned}
\textrm{meas}(S^-\cap(h_1,b))\leq\delta_1 \qquad \mbox{and} \qquad
	\textrm{meas}(S^+\cap(h_1,h_1+2\delta_1))\geq\delta_1.
	\end{aligned}
\end{equation*}
Furthermore, from \eqref{E_p(w)<Nc0+eps^l}  with $m=1$, we obtain
\begin{equation*}
	\int_{S^+\cap(h_1,h_1+2\delta_1)}\frac{F(w)}\varepsilon\, dx\leq c_p+1,
\end{equation*}
and therefore there exists $y_1\in S^+\cap(h_1,h_1+2\delta_1)$ such that
\begin{equation*}
	F(w(y_1))\leq C\varepsilon, \qquad \quad C=\frac{c_p+1}{\delta_1}.
\end{equation*}
From the definition \eqref{defF2}, it follows that $w(y_1)\geq 1-C_1\e^{\frac1{n}}$.
The existence of $x_1{\in S^-\cap(h_1-2\delta_1,h_1)}$ such that $w(x_1)\leq-1+C_1\varepsilon^\frac1{n}$ can be proved similarly.

Now, let us prove that \eqref{x_1,y_1} implies \eqref{E_p(w)-xk,yk} in the case $m=1$,
and as a trivial consequence we obtain the statement \eqref{E_p(w)>Nc0-eps^l} with $m=1$ and $N=1$.
Indeed, proceeding as in \eqref{eq:c_eps} one has
\begin{equation}\label{stima-l=1}
	\begin{aligned}
		E_\e[w]&\geq\int_{x_1}^{y_1}\left[\frac{\varepsilon^{p-1}}p|w_x|^p+\frac{F(w)}\varepsilon\right]dx
		\geq\left(\frac{p}{p-1}\right)^{\frac{p-1}{p}}\int_{x_1}^{y_1} |w_x|F(w)^{\frac{p-1}{p}}dx\\
		&\geq c_p-\left(\frac{p}{p-1}\right)^{\frac{p-1}{p}}\int_{1-C_1\varepsilon^\frac1{n}}^1 F(s)^{\frac{p-1}{p}}ds \\
		&\qquad  -\left(\frac{p}{p-1}\right)^{\frac{p-1}{p}}\int_{-1}^{-1+C_1\varepsilon^\frac1{n}}F(s)^{\frac{p-1}{p}}ds \\
		&\geq c_p-C_1\varepsilon^\alpha,
	\end{aligned}
\end{equation}
where $\alpha$ is defined in \eqref{eq:exp_alg}.
This concludes the proof in the case $m=1$ with one transition ($N=1$).

We now enter the core of the induction argument, proving that if \eqref{E_p(w)-xk,yk} holds true for for any $j\in\{1,\dots,m-1\}$, $m\geq2$, then \eqref{x_k,y_k} holds true.
By using \eqref{|w-v|_L^1<delta_l} we have
\begin{equation}\label{meas>delta_l-k}
	\textrm{meas}(S^+\cap(y_{m-1},y_{m-1}+2\delta_m))\geq\delta_m.
\end{equation}
Furthermore, by using  \eqref{E_p(w)<Nc0+eps^l} and \eqref{E_p(w)-xk,yk} in the case $m-1$, we deduce
\begin{equation*}
	\int_{y_{m-1}}^b\frac{F(w)}\varepsilon dx\leq C_m\varepsilon^{k_m},
\end{equation*}
implying
\begin{equation}\label{int F(w)<Ceps^k+1}
	\int_{S^+\cap(y_{m-1},y_{m-1}+2\delta_m)} F(w)\,dx\leq C_m\varepsilon^{k_m+1}.
\end{equation}
Finally, from  \eqref{meas>delta_l-k} and \eqref{int F(w)<Ceps^k+1} there exists $y_{m}\in S^+\cap(y_{m-1},y_{m-1}+2\delta_m)$ such that
\begin{equation*}
	F(w(y_{m}))\leq \frac{C_m}{\delta_m}\varepsilon^{k_m+1},
\end{equation*}
and, as a consequence, we have the existence of $y_{m}\in(y_{m-1},y_{m-1}+2\delta_m)$ as in \eqref{x_k,y_k}. 
The existence of $x_{m}\in(x_{m-1}-2\delta_m,x_{m-1})$ can be proved similarly. 

Reasoning as in \eqref{stima-l=1}, one can easily check that \eqref{x_k,y_k} implies
\begin{equation*}
	\int_{x_{m}}^{y_{m}}\left[\frac{\varepsilon^{p-1}}p |w_x|^p+\frac{F(w)}\varepsilon\right]dx\geq c_p-C_m\varepsilon^{k_{m+1}},
\end{equation*}
and  the induction argument is completed, as well as the proof in case $N=1$.

The previous argument can be easily adapted to the case $N>1$. 
Let $v$ be as in \eqref{vstruct}, and set $a=h_0, h_{N+1}=b$. 
We argue as in the case $N=1$ in each point of discontinuity $h_i$, by choosing
the constant $\delta_m$ so that 
$$
	{h_i}+2m\delta_m<h_{i+1}-2m\delta_m,\qquad \quad 0\leq i\leq N,
$$
and by assuming, without loss of generality, that $v=-1$ on $(a,h_1)$. 
Proceeding as in \eqref{x_1,y_1}, one can obtain the existence of $x^i_1\in(h_i-2\delta_m,h_i)$ and $y^i_1\in(h_i,h_i+2\delta_m)$ such that
\begin{align*}
	w(x^i_1)&\approx (-1)^i, &w(y^i_1)&\approx(-1)^{i+1},\\
	F(w(x^i_1))&\leq C\varepsilon, &F(w(y^i_1))&\leq C\varepsilon.
\end{align*}
On each interval $(x_1^i,y_1^i)$ we estimate as in \eqref{stima-l=1}, so that by summing  one obtains
$$
	\sum_{i=1}^N\int_{x_1^i}^{y_1^i}\left[\frac{\varepsilon^{p-1}}p |w_x|^p+\frac{F(w)}\varepsilon\right]dx\geq Nc_p-C_m\varepsilon^\alpha,
$$
that is \eqref{E_p(w)>Nc0-eps^l} with $m=1$. 
Arguing inductively as done in the case $N=1$, we obtain \eqref{E_p(w)>Nc0-eps^l} for the general case $m\geq2$.
\end{proof}

\begin{rem}\label{sharp}
It is worth noticing that the assumption $n>p$ implies that the sequence \eqref{eq:exp_alg} is increasing, bounded from above and, as a consequence, the best exponent we can obtain is
\begin{equation}\label{eq:limitalpha}
	\gamma_{n,p}:=\lim_{m\to+\infty}k_m=\frac{\alpha}{1-\alpha}=\frac{np}{n-p}-1.
\end{equation}
In particular, proceeding as in Proposition \ref{prop:ex-met} (see also Remark \ref{rem:translayer}) 
we can construct a function $w$ satisfying \eqref{|w-v|_L^1<delta_l} and $E_\e[w]<Nc_p$, implying that  
$$0<Nc_p-E_\e[w]\leq C_m\e^{k_m}, \qquad \mbox{for any} \ \ m \in \N.$$ 
Hence, we only have an \emph{algebraic small reminder} of order $\mathcal{O}(\e^{k_m})$. 
As we have already mentioned, the exponent $\gamma_{n,p}$ {obtained in the limit \eqref{eq:limitalpha} is the same of the} one obtained in \cite{BetSme2013}, which, in the case $p=2$, is $\gamma=1+{4/(n-2)}$, and in particular, when $n={2}$ (that is, we consider the usual Allen--Cahn equation \eqref{eq:Al-Ca}), $\gamma_{n,p}$ turns out to be equal to $+\infty$, which is what we expect after the results of \cite{BrKo90}.

On the other hand, the procedure we used in the proof of Proposition \ref{prop:lower_deg} can be also applied to the case $n\leq p$.
To be more precise, in the case $n=p$, we have $\alpha=1$ and the sequence \eqref{eq:exp_alg} is simply $k_m:=m-1$.
Thus, we obtain an algebraic small reminder $\mathcal{O}(\e^k)$ for any $k\in\mathbb N$ and
we generalize to the case $p>1$ the result of \cite{BrKo90}. The same can be obtained if considering the limit $p \to n$ in \eqref{eq:limitalpha}, which gives  $\displaystyle{\lim_{p \to n} \gamma_{n,p} = +\infty}$. However, let us underline again that in this case the sharp estimate is given by \eqref{eq:lower}, providing an {\it exponentially small reminder}. 

Finally, if $n<p$, one has $\alpha>1$ and the sequence \eqref{eq:exp_alg} diverges to $+\infty$ as $m\to+\infty$ (notice that here $k_m$ is not necessary a natural number),
meaning that the exponent of $\e$ in \eqref{E_p(w)-xk,yk} can be chosen arbitrarily large.
\end{rem}

Proposition \ref{prop:lower_deg} is the key point to prove the algebraic slow motion of the solutions in the case $n<p$.
As done in Section \ref{sec:2n=p}, we fix a piecewise constant function $v$ with $N$ transitions as in \eqref{vstruct} and we assume that the initial datum $u^\e_0$ satisfies
\begin{equation}\label{eq:ass-u0-deg}
	\lim_{\varepsilon\rightarrow 0} \|u^\varepsilon_0-v\|_{{}_{L^1}}=0,
\end{equation}
and that there exists $m\in\mathbb N$ such that
\begin{equation}\label{eq:energy-ini-deg}
	E_\varepsilon[u^\varepsilon]\leq Nc_p+\e^{k_m},
\end{equation}
for any $\varepsilon\ll1$, where the energy $E_\e$ and the positive constants $c_ p, k_m$ are defined in \eqref{eq:energy}, \eqref{eq:c_eps} and \eqref{eq:exp_alg}, respectively.

The main result of this section is the following theorem.
\begin{thm}[algebraic slow motion with a $p$-Laplacian diffusion in the degenerate or supercritical case $n>p$]\label{thm:main2}
Let $u^\varepsilon$ be the solution to \eqref{eq:P-model}-\eqref{eq:Neu}-\eqref{eq:initial} with $n>p$ 
and with initial datum $u_0^{\varepsilon}$ satisfying \eqref{eq:ass-u0-deg}-\eqref{eq:energy-ini-deg}. 
Then,
\begin{equation}\label{eq:limit-deg}
	\sup_{0\leq t\leq l{\e^{-k_m}}}\|u^\varepsilon(\cdot,t)-v\|_{{}_{L^1}}\xrightarrow[\varepsilon\rightarrow0]{}0,
\end{equation}
 for any $l>0$.
\end{thm}

\begin{proof}
The proof follows the same steps of the proof of Theorem \ref{thm:main} 
and it is obtained by using Proposition \ref{prop:lower_deg} instead of Proposition \ref{prop:lower}.
In particular, proceeding as in the proof of Proposition \ref{prop:L2-norm}, one can prove that
there exist $\varepsilon_0, C_1, C_2>0$ (independent on $\varepsilon$) such that
\begin{equation*}
	\int_0^{C_1\varepsilon^{-(k_m+1)}}\|u_t^\varepsilon\|^2_{{}_{L^2}}dt\leq C_2\varepsilon^{k_m+1},
\end{equation*}
for all $\varepsilon\in(0,\varepsilon_0)$.
Thanks to the latter estimate, we can prove \eqref{eq:limit-deg} in the same way we proved \eqref{eq:limit} (see \eqref{trianglebar} and the following discussion).
\end{proof}

\section{Layer dynamics in the case $n\geq p$}\label{LD}
Theorems \ref{thm:main} and \ref{thm:main2} show that  solutions to \eqref{eq:P-model}-\eqref{eq:Neu}  arising from initial data with $N$-transition layers maintain such unstable structure for long times (precisely, {\it exponentially} long times or {\it algebraically} long times if $n=p$ or $n>p$, respectively). 
These results are tantamount to a precise description of the motion of the transition points $h_1, \dots , h_N$, showing that they move with a very small velocity as $\e\to0^+$.

Following the strategy of \cite{FLM19, FPS1, Grnt95}, let us consider $v$ a piecewise constance function as in \eqref{vstruct}, and $u: [a,b] \to \mathbb{R}$ an arbitrary function. 
We define their {\it interfaces} as follows:
\begin{equation*}
I[v]=\{h_1, \dots , h_N\} \qquad \mbox{and} \qquad I_K[u]=u^{-1}(K),
\end{equation*}
where $K \subset \mathbb{R}\setminus \{ \pm 1\}$ is an arbitrary closed subset. Also, for any $A,B \subset \mathbb{R}$, we define 
\begin{equation*}
d(A,B):=\max \left \{ \sup_{\alpha \in A} d(\alpha,B), \sup_{\beta \in B} d(\beta,A)\right \},
\end{equation*}
where $d(\beta,A):=\inf \{|\beta-\alpha| \, : \,  \alpha \in A \}$.

The next  lemma  shows that  the distance between the interfaces $I_K[u]$
and $I[v]$ is small, providing some  smallness assumptions on the $L^1$--norm of the difference $ u-v$ and on the energy $E_\e[u].$
The result is purely variational in character and holds true both in the critical ($n=p$) and degenerate ($n>p$) cases.

\begin{lem}\label{lemma:layermotion}
Let $F$ as in \eqref{defF}, $v$ as in \eqref{vstruct} and $r$ such that \eqref{eq:r} holds. 
Given $\delta \in (0,r)$, there exist constants $\hat \delta, \e_0, \Gamma >0$ such that, if $u \in H^1([a,b])$ satisfies
\begin{equation}\label{ipolemma}
	\| u- v \|_{L^1} < \hat \delta \qquad \mbox{and} \qquad  E_\e[u] \leq Nc_p + \Gamma,
\end{equation}
then, for any $\e \in (0,\e_0)$ there holds
\begin{equation}\label{tesilemma}
d(I_K[u],I[v])\leq \frac{\delta}{2}.
\end{equation}
\end{lem}

\begin{proof}

The proof can be done by contradiction as in   \cite{ FPS1,Grnt95}, provided to use the following lower bound on the energy
\begin{equation*}
	E_\e[{u}]\geq\left(\frac{p}{p-1}\right)^{\frac{p-1}{p}}\int_{-1}^{+1} {F(s)}^\frac{p-1}{p}\,ds=c_p.
\end{equation*}
We thus refer the reader's to  see, for instance, the proof of \cite[Lemma 3.5]{FPS1}.
\end{proof}

We are now ready to prove the following result concerning the slow motion of the transition points $h_1, \dots , h_N$, showing that  they evolve {\it exponentially} or {\it algebraically} slowly if $n=p$ or $n>p$, respectively.
 
\begin{thm}\label{thm:main3}
Let $F$ as in \eqref{defF} and let $u^\e$ be the solution to  \eqref{eq:P-model}-\eqref{eq:Neu} with initial datum $u^\e_0$ satisfying $\displaystyle{\lim_{\varepsilon\rightarrow 0} \|u^\varepsilon_0-v\|_{{}_{L^1}}=0}$ and \eqref{eq:energy-ini} in the case $n=p$ or \eqref{eq:energy-ini-deg} in the case $n>p$. 
Given $\delta \in (0,r)$, set
\begin{equation*}
t_\e(\delta) = \inf \{ t  :  d(I_K[u_\e(\cdot, t)], I_K[u_0^\e]) > \delta \},
\end{equation*}
where $K \subset \mathbb{R} \setminus \{ \pm 1\}$. 
Then there exists $\e_0>0$ such that, if $\e \in (0,\e_0)$
\begin{equation*}
	t_\e(\delta) > \omega(\e),
\end{equation*}
where
\begin{equation*}
	\omega(\e):= \left\{ \begin{aligned}
		&\exp (Ap/2\e) \qquad &\mbox{if} \quad n=p, \\
		&\e^{-k_m},  \qquad &\mbox{if} \quad n>p,
	\end{aligned}\right.
\end{equation*}
with $A$ and $k_m$ appearing in Theorems \ref{thm:main} and \ref{thm:main2}.
\end{thm}

\begin{proof}
We start with the case $n=p$. 
We choose $\e_0$ small enough such that the assumption on $u_0^\e$ implies that \eqref{ipolemma} is satisfied; 
hence, from Lemma \ref{lemma:layermotion} it follows that
\begin{equation*}
	d(I_K[u_0^\e],I[v]) < \frac{\delta}{2}.
\end{equation*}
Also, if considering the time dependent solution $u^\e(\cdot, t)$, from \eqref{eq:limit} in Theorem \ref{thm:main} and since $E_\e[u]$ is a non-increasing function of $t$, 
it follows that \eqref{ipolemma} is satisfied for $u^\e(\cdot, t)$, for any $t < \exp (Ap/2\e)$, implying \eqref{tesilemma} holds for $u^\e$ as well. 
As a consequence, from the triangular inequality, we have
\begin{equation*}
	d(I_K[u^\e(t)],I_K[u_0^\e]) <\delta,
\end{equation*}
for all $t \in (0, \exp (Ap/2\e))$. 

When $n>p$ we can proceed with the exact same computations by making use of \eqref{eq:limit-deg} in Theorem \ref{thm:main2}; 
we thus end up with 
\begin{equation*}
	d(I_K[u^\e(t)],I_K[u_0^\e]) <\delta,
\end{equation*}
for all $t \in (0,\e^{-k_m})$, and the proof is complete.
\end{proof}

Theorem \ref{thm:main3}, together with Theorems \ref{thm:main} and \ref{thm:main2}, prove that solutions to \eqref{eq:P-model}-\eqref{eq:Neu} with a transition layer structure evolve {\it exponentially slowly} in the case $n=p$ and {\it algebraically slowly} if $n>p$; 
they indeed maintain the same profile of their initial datum for times of  $\mathcal O(\exp(Ap/2\e))$ and $\mathcal O(\e^{-k_m})$ respectively, 
and the transition points move with exponentially (algebraically respectively) small speed.
\vskip0.1cm
We conclude this paper with some numerical simulations showing the rigorous results of Sections \ref{sec:2n=p}, \ref{sec:degcase} and \ref{LD}. In Figure \ref{fig:2n=p} we compare the solutions to \eqref{eq:P-model} in the cases $p=2$ (left) and $p=4$ (right); in both figures the potential  $F$ is as in \eqref{eq:2n=p-model} and $\e=0.1$. 
It is interesting to see how in the case $p=n=2$ (corresponding to the classical Allen--Cahn equation \eqref{eq:Al-Ca}) 
the solution maintains the same transition layer structure of the initial datum up to $t=2*10^4$, 
and that when $t=3*10^4$ the first transition points collapse; 
on the other hand, when $p=4$ (right picture) the solutions is still almost indistinguishable from $u_0$ for $t>10^4$, and it is only for $t=7*10^8$ that the layers disappear. 
This is consistent with the estimate \eqref{eq:limit} proven in Theorem \ref{thm:main}, since $e^{Ap/2\e}>e^{1/\e}\approx2.2*10^4$ whenever $p>2$.

\begin{figure}[ht]
\centering
\includegraphics[width=6cm,height=5cm]{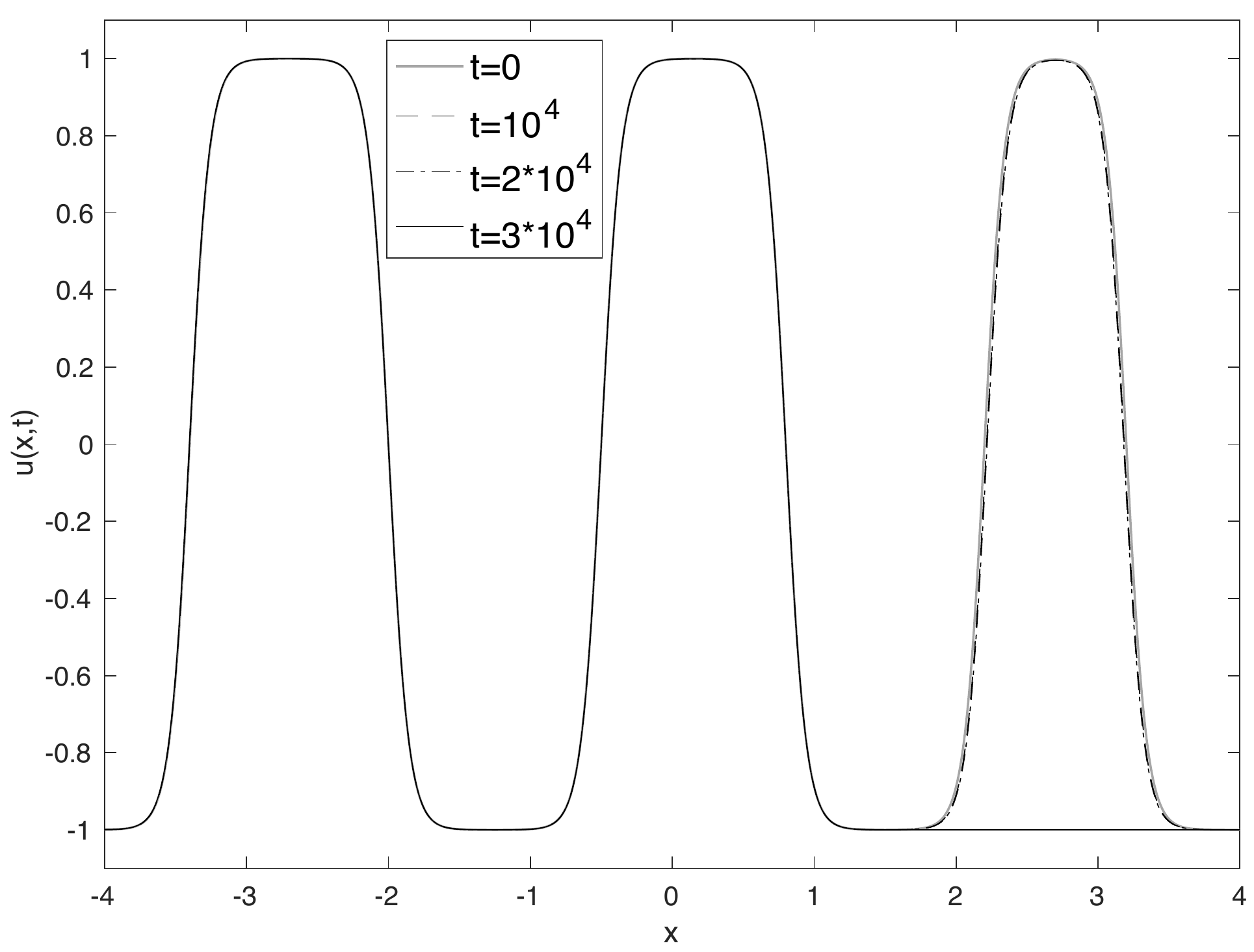}
\quad
\includegraphics[width=6cm,height=5cm]{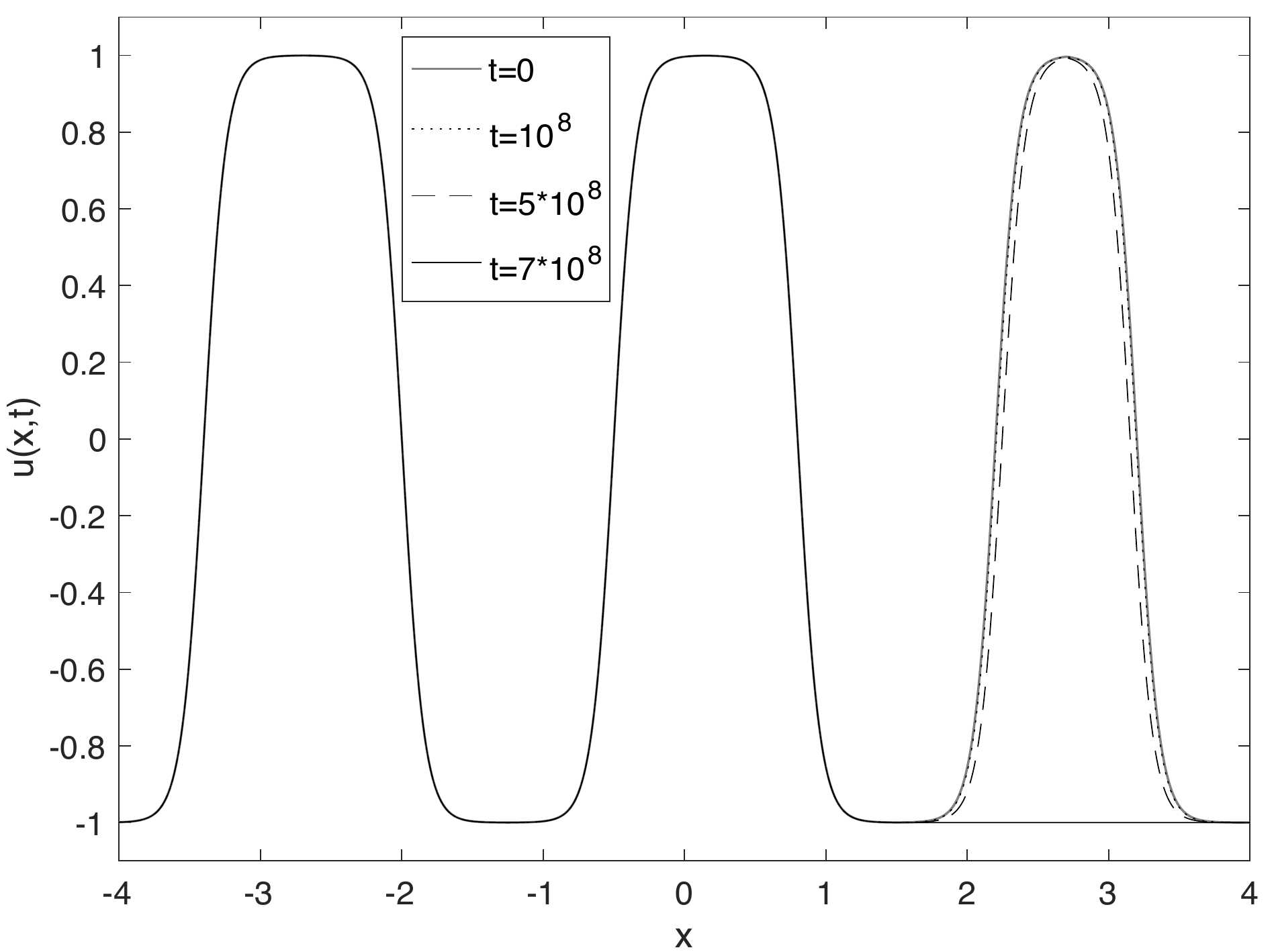}
\hspace{3mm}
\caption{\small{ Solutions to \eqref{eq:P-model} for  $\e=0.1$, $p=n=2$ (left), and $p=n=4$ (right);
the initial datum $u_0$ is as in \eqref{eq:translayer} with $6$ transition points located at $(-3.4,-2,-0.5, 0.8, 2.2, 3.2).$ }}
\label{fig:2n=p}
\end{figure}

Figure \ref{fig:2n>p} shows the solutions to \eqref{eq:P-model} with $n>p$, choosing $p=2$, $n=4$ in the left picture and $p=3$, $n=4$ in the right one;  in such a case we only have an algebraically slow motion of the solutions, and the layers collapse for times which are much smaller if compared to the ones of Figure \ref{fig:2n=p} (compare, for instance, the times in the two pictures on the right hand side, where we chose the same potential $F$ and we only change the value of $p$).
It is also interesting to compare the pictures on the left hand side in Figures \ref{fig:2n=p} and \ref{fig:2n>p}: in both cases we consider the classical Laplacian operator with $\e=0.1$, and we only change the potential $F$, in order to switch from the critical case $n=p$ to the degenerate case $n>p$. 
It is glaring how the time taken by the first interfaces to disappear is only algebraical in the second case, where we see the first transition points to collapse for $t=800$.

\begin{figure}[ht]
\centering
\includegraphics[width=6cm,height=5cm]{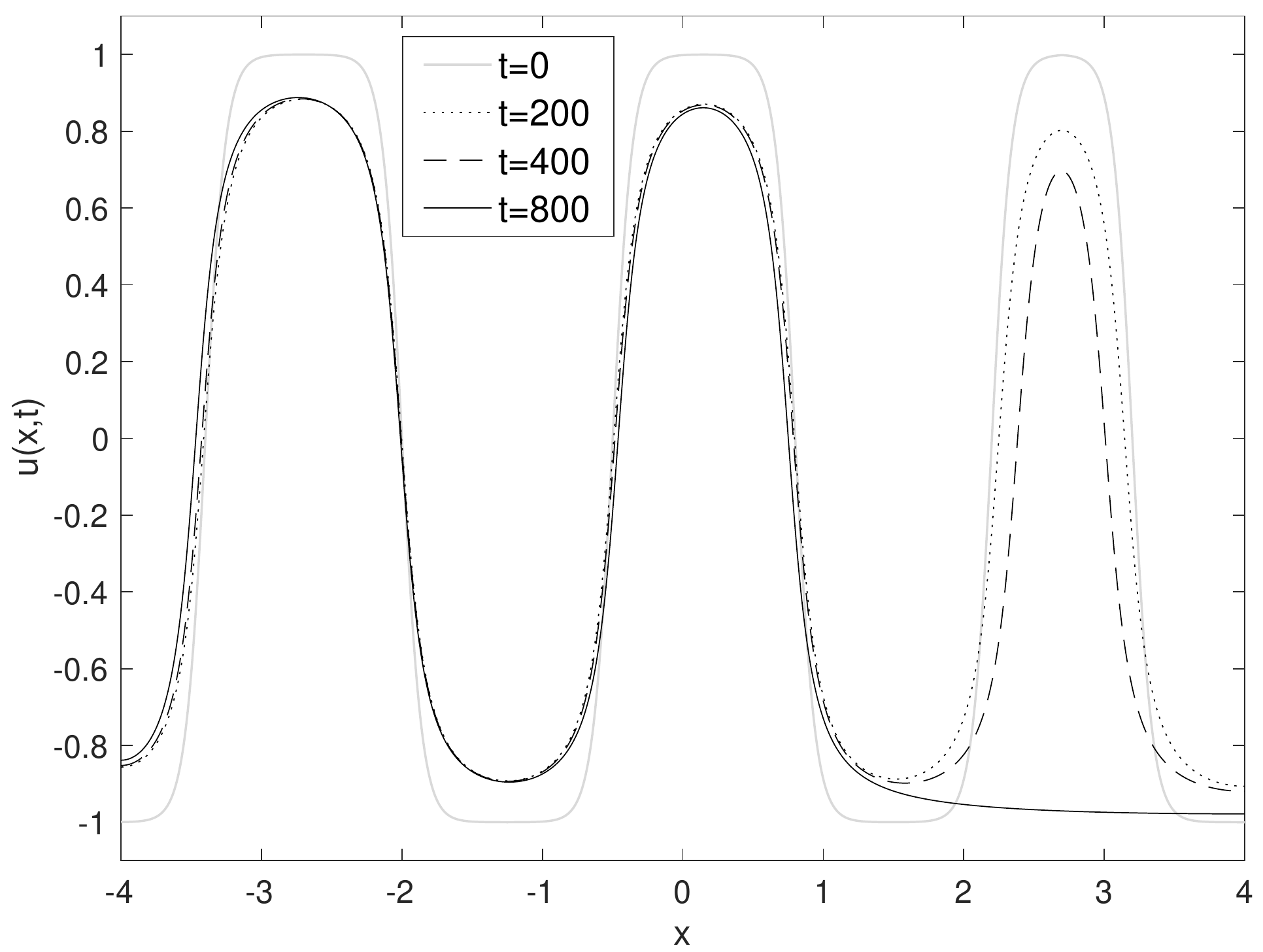}
\quad
\includegraphics[width=6cm,height=5cm]{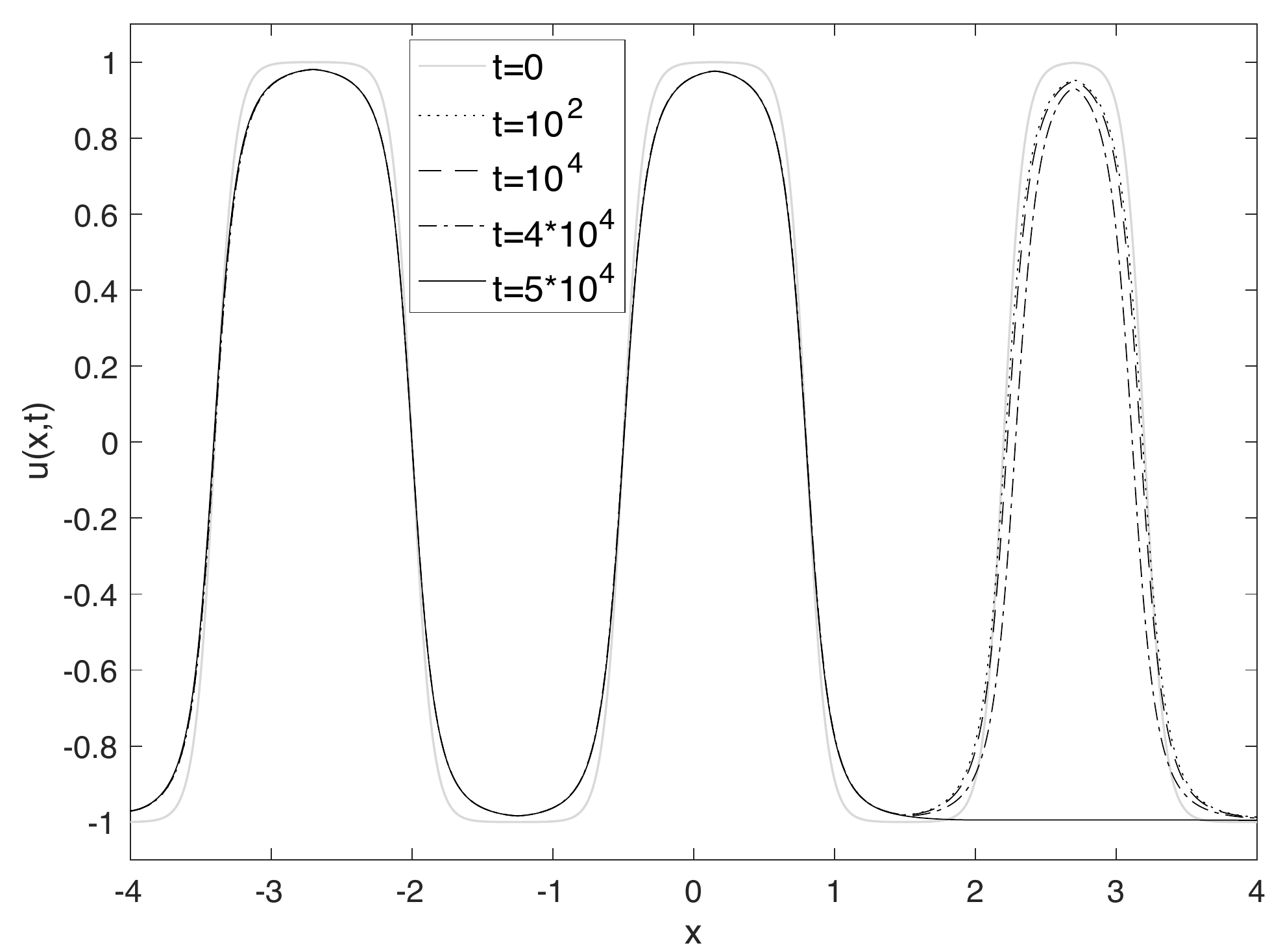}
\hspace{3mm}
\caption{\small{Solutions to \eqref{eq:P-model} for $\e=0.1$,  $p=2$ (left)  and $p=3$ (right); the initial datum $u_0$ is as in Figure \ref{fig:2n=p}, while the potential $F$ is as in \eqref{defF2}  with $n=4$.}}
\label{fig:2n>p}
\end{figure}

As a last example, we consider the case $n>p$ with $p$ a given {\it real} number. 
In Figure \ref{fig:p5.5} we take $n=8$, $p=\pi$ (left picture) and $p=5.5$ (right picture); 
the first interfaces vanish respectively for  $t=2*10^4$ and $t=5*10^9$, 
while to see the collapsing of another transition point we have to wait till $t=2*10^5$ if $p=\pi$ and $t=5*10^{10}$ if $p=5.5$.  
We observe once again that the bigger is $p$, the longer is the time to see the annihilation of the interfaces.

\begin{figure}[ht]
\centering
\includegraphics[width=6cm,height=4.9cm]{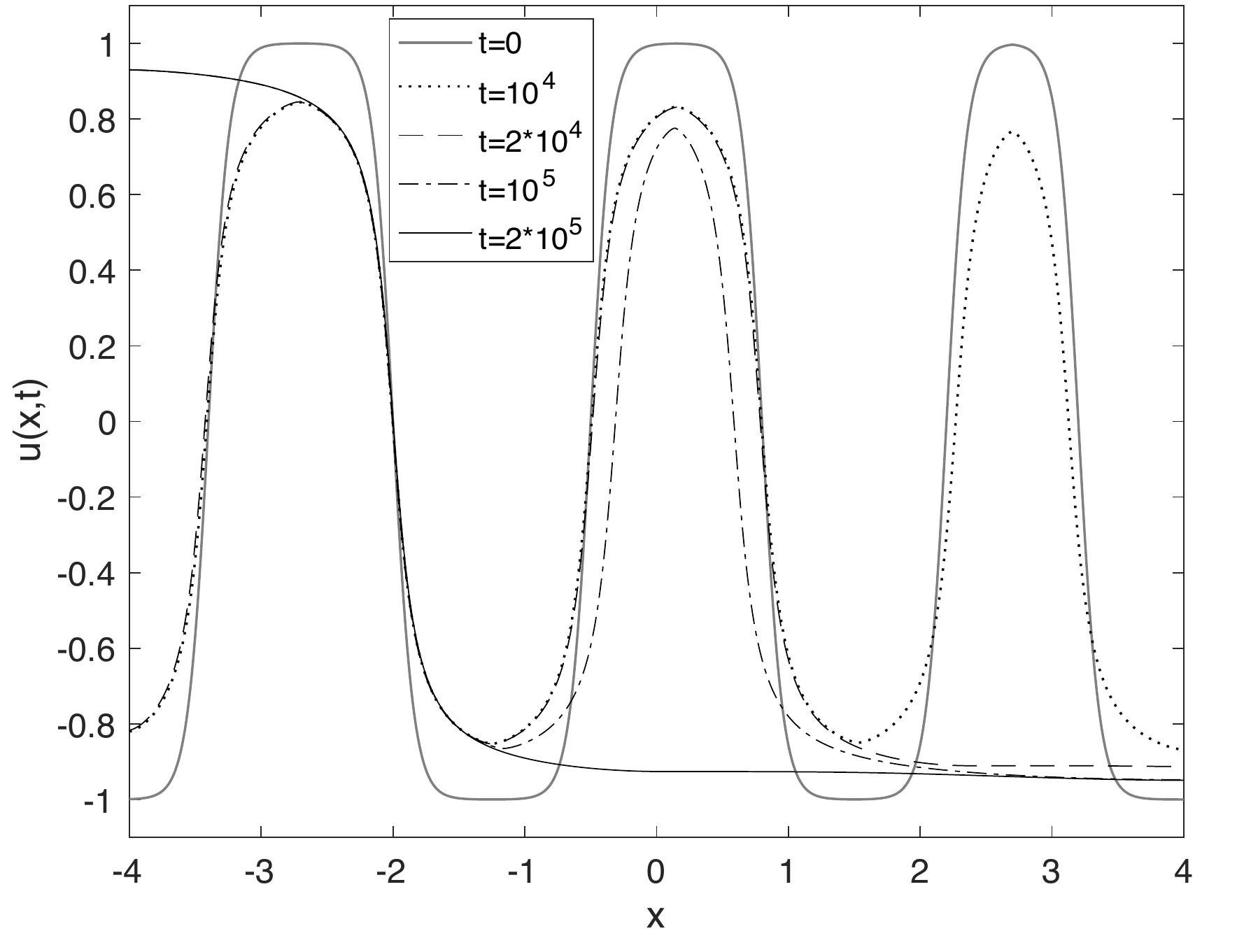}
\quad
\includegraphics[width=6cm,height=5cm]{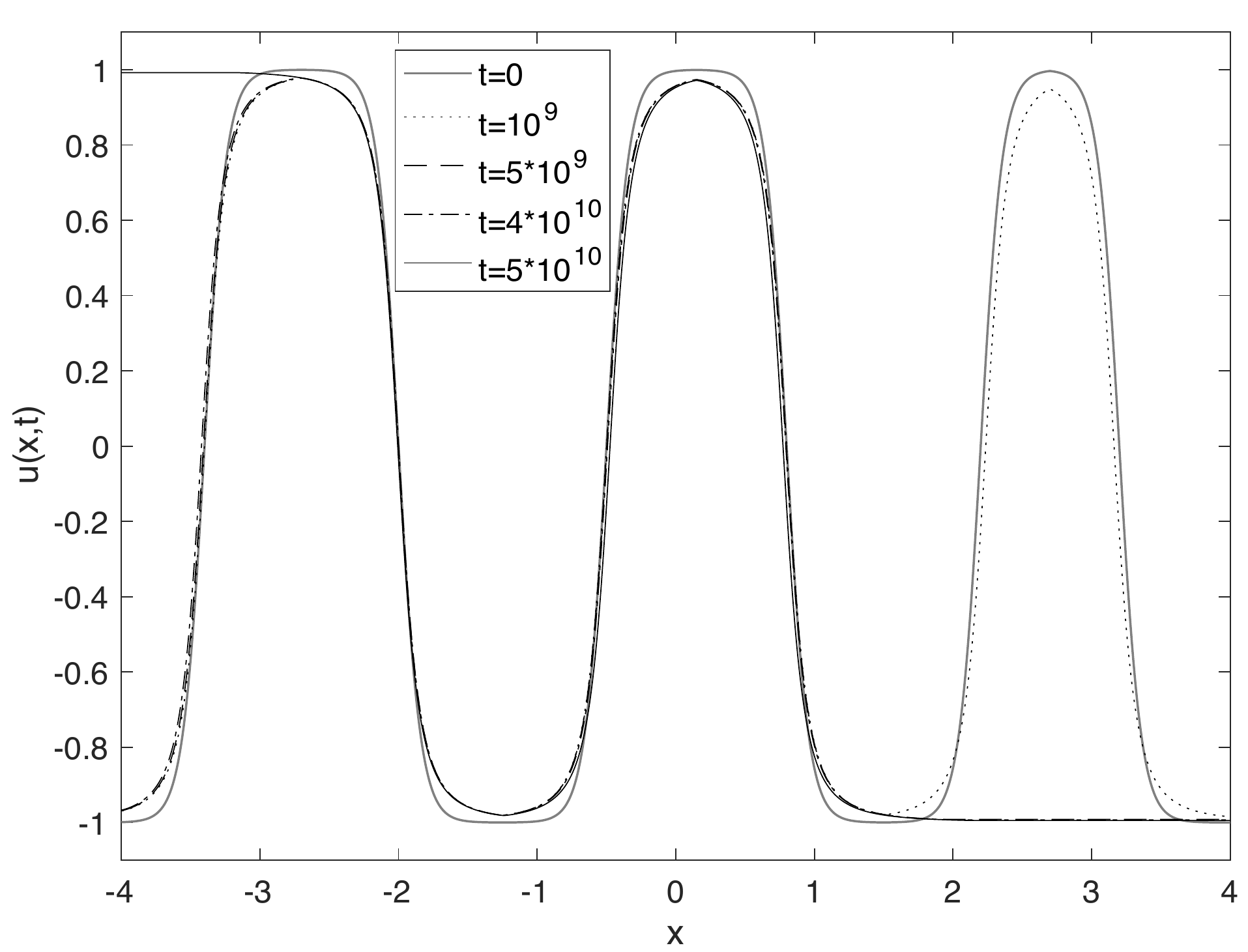}
\hspace{3mm}
\caption{\small{Solutions to \eqref{eq:P-model} for $\e=0.1$, $n=8$ and $p=\pi$ (left) and $p=5.5$ (right); the initial datum $u_0$ is as in Figure \ref{fig:2n=p}.}}
\label{fig:p5.5}
\end{figure}

\section*{Acknowledgements}

The work of RGP was partially supported by DGAPA-UNAM, program PAPIIT, grant IN-100318.

\bibliography{riferimenti2}

%This was the new (SIAM based) style for the bibliography
\bibliographystyle{newstyle}

\end{document}